\documentclass{scrartcl}
\usepackage{amsmath,amsfonts,amssymb,amsthm,mathtools}
\usepackage{mathrsfs}
\usepackage{xcolor}
\usepackage{dsfont}
\usepackage{hyperref}
\usepackage{enumitem}
\usepackage[textsize=scriptsize]{todonotes}
\usepackage{comment}
\theoremstyle{definition}
 \newtheorem{dfn}{Definition}[section]
 \newtheorem{remark}[dfn]{Remark}  
\theoremstyle{plain}
 \newtheorem{thm}[dfn]{Theorem}
 
 \newtheorem{lem}[dfn]{Lemma}
 \newtheorem{cor}[dfn]{Corollary}

\numberwithin{equation}{section}
\newcommand{\R}{\mathbb{R}}

\newcommand{\N}{\mathbb{N}}

\newcommand{\dd}{{\mathrm d}}
\newcommand{\de}{{\mathrm d}}
\DeclareMathOperator{\dv}{div}

\DeclareMathOperator{\spa}{span}

\newcommand{\bbY}{\mathbb{Y}}

\newcommand{\calC}{\mathcal{C}}

\newcommand{\calL}{\mathcal{L}}
\newcommand{\calM}{\mathcal{M}}

\newcommand{\calP}{\mathcal{P}}

\newcommand{\calS}{\mathcal{S}}

\newcommand{\calV}{\mathcal{V}}

\newcommand{\frakR}{\mathfrak R}
\newcommand{\ov}[1]{\overline{#1}}

\newcommand{\norm}[1]{\lVert #1 \rVert}

\newcommand{\setcl}[2]{\big\{#1 \,\vert\, #2\big\}}
\newcommand{\setcL}[2]{\bigg\{#1 \,\Big\vert\, #2\bigg\}}
\newcommand{\BV}{\mathrm{BV}}
\newcommand{\f}[1]{{{ #1}}}

\newcommand{\E}{\mathcal{E}}
\newcommand{\LRsigma}[1]{L^{#1}_\sigma}
\DeclareMathOperator{\di}{div}
\newcommand{\sym}{\text{sym}}
\renewcommand{\O}{{\Omega}}
\DeclareMathOperator{\tr}{tr}

\newcommand{\EEE}{\color{black}}

\begin{document}
\title{On the equivalence of generalized solution concepts for systems of hyperbolic conservations laws in fluid dynamics}

\author{Thomas Eiter\thanks{Freie Universit\"at Berlin, Department of Mathematics and Computer Science, Arnimallee 14, 14195 Berlin, Germany; Weierstrass Institute for Applied Analysis and Stochastics, 
Anton-Wilhelm-Amo-Str. 39, 10117 Berlin, Germany,
\texttt{thomas.eiter@wias-berlin.de}}
\and
Robert Lasarzik\thanks{Weierstrass Institute for Applied Analysis and Stochastics, 
Anton-Wilhelm-Amo-Str. 39, 10117 Berlin, Germany,
\texttt{robert.lasarzik@wias-berlin.de}}
\and
Emil Wiedemann\thanks{
Department of Mathematics, Friedrich-Alexander-Universit\"{a}t Erlangen-N\"{u}rnberg, Cauerstr. 11, 91058 Erlangen, Germany,
\texttt{emil.wiedemann@fau.de}}
}
\date{\today}

\maketitle

 \begin{abstract}
We investigate the relation between several generalized solution concepts for nonlinear PDE systems from fluid dynamics. More precisely, we study measure-valued solutions, dissipative weak solutions, and energy-variational solutions. For the incompressible Euler equations, we prove the equivalence of all three concepts, provided that the energy inequality is formulated in the appropriate way. For several important examples of conservation laws arising in fluid dynamics, we establish the equivalence between energy-variational and suitably refined dissipative weak solutions, where the defect measures are controlled sharply by the energy defect. These examples comprise the compressible isentropic Euler system, the Euler--Korteweg system, and the Euler--Poisson system. 
\end{abstract}


\section{Introduction}

Generalized solutions are omnipresent in the theory of partial differential equations, especially fluid dynamics. It is well known that in general, smooth solutions do not exist globally in time, as shocks or even turbulence may arise.
The lack of compactness and the construction of infinitely many weak solutions to the incompressible Euler~\cite{delellis} and Navier--Stokes equations~\cite{buckmaster} suggest that there is an inherent loss of regularity. To deal with this issue, especially when proving global-in-time existence, there is a plethora of generalized solution concepts.
Even before Leray~\cite{leray1934} introduced his ``turbulent'', now called weak, solutions to the Navier--Stokes equations, other authors used generalized formulations for the investigation of different PDEs. 
Already Lagrange introduced a weak formulation for the one-dimensional clamped plate~\cite{lagrange}, and Poincar{\'e} considered a very weak formulation of the Poisson equation with Robin boundary conditions~\cite{poincare}. 
These concepts became very powerful some decades later when they were combined with functional analytic tools like weak convergence~\cite{wiener}. 

Nowadays weak solutions are the standard concept for nonlinear PDEs, but they have certain drawbacks. Many PDE systems lack the required compactness properties in order to pass to the limit in the weak formulation, and convex integration allows to construct infinitely many weak solutions. 
Therefore, more general solvability concepts are of interest.
For instance, measure-valued solutions were first introduced in the context of PDEs by Tartar~\cite{tartar} as an in-between step in the proof of the famous div-curl lemma. Measure-valued solutions as a proper solvability framework were first defined by DiPerna~\cite{MR775191} and DiPerna--Majda~\cite{DiPernaMajda} and have been investigated for different PDE systems since then, like the compressible Navier--Stokes system~\cite{NavCompMeas}, the Euler--Korteweg system~\cite{Korteweg} or the complete Euler system~\cite{CompletEuler}.
A different solvability concept based on a relative energy inequality
is due to Lions~\cite{lionsbook}, who introduced the notion of dissipative solutions, in particular for the limit passage from the Boltzmann equation to the Euler equations.
In~\cite{weakstrongeuler}, a weak-strong uniqueness principle for measure-valued solutions was shown, and the expectation of a measure-valued solution was identified as a dissipative solution. For many systems from fluid dynamics, the notion of dissipative weak solutions has become an indispensable solution concept. These solutions have been investigated, for instance, for the compressible Euler equations~\cite{Fereisl21_NoteLongTimeBehaviorDissSolEuler}, the Euler--Korteweg system~\cite{GGSW}, and the complete Euler system~\cite{CompletEuler}.

Recently, the concept of energy-variational solutions has been introduced for general hyperbolic conservation laws~\cite{eiterlasarzik2024envar} and general evolutionary PDE systems~\cite{abramo,Marcel}. 
This concept  allows to prove global-in-time existence for a relatively large class of systems via a minimizing-movement scheme. 
For the Euler equations for
incompressible and for compressible isentropic fluids, 
it was shown that these solutions are equivalent to dissipative weak solutions~\cite{eiterlasarzik2024envar}. This observation was recently extended to other systems of PDEs~\cite{RobVari}. For instance, in the case of the two-phase Navier--Stokes system, energy-variational solutions are equivalent to varifold solutions (\textit{cf.}~\cite{abels}), and in the case of the hyperbolic system modeling polyconvex elastodynamics, energy-variational solutions were shown to be equivalent to the already established Young measure-valued solutions (\textit{cf.}~\cite{demoulini}).
Moreover, energy-variational solutions were identified as a valuable technique for identifying the limits of solutions to structure-inheriting numerical schemes~\cite{LasarzikReiter2023,eitergiesselmann2026korteweg}. 

As all the considered generalized solution concepts in this area of nonlinear evolutionary PDEs exhibit multiple possibly nonunique solutions, the identification of appropriate selection criteria is important to select reasonable solutions under the multitude of possible generalized solutions. 
This question has recently been investigated for measure-valued~\cite{EmilMaiximalTurb}, dissipative weak~\cite{FeireislNew}, and energy-variational solutions~\cite{Marcel}.

So far, however, the question to what extent these solution concepts are related or even equivalent has not been fully answered.
In this work, we intend to close this gap by showing certain equivalence results between generalized solution concepts for different PDE systems from fluid dynamics.  

In the first part of the paper,
we consider the three most relevant contemporary ``very weak'' solution concepts for the incompressible Euler equations: measure-valued solutions, dissipative weak solutions and energy-variational solutions. 
The equivalence between the latter two was already shown in~\cite{eiterlasarzik2024envar}, while the equivalence of these concepts with measure-valued solutions is often taken for granted.
However, to the best of the authors’ knowledge, this has not been shown anywhere so far.
Using the appropriate formulation of the energy inequality (see Remark~\ref{rem:en.ineq} below),
we prove that all three concepts are equivalent.

We mention in passing two other related solution concepts that have been developed in the context of the incompressible Euler equations: The above-mentioned dissipative solutions of P.-L.~Lions~\cite{lionsbook} and the subsolutions from convex integration theory~\cite{delellis}. The former constitute a broader solution class compared with the solutions considered here: Every energy-variational or dissipative weak solution and every mean of a measure-valued solution (see Definitions~\ref{def:envarIncommp},~\ref{def:disssol.incomp} and~\ref{def:meas.incomp}) is a dissipative solution in the sense of Lions, but the converse is not expected (cf.~\cite[Remark 3.7]{abramo}). On the other hand, the subsolutions in convex integration or, in different terminology, the solutions of an Euler--Reynolds system are formally equivalent with what we call dissipative weak solutions; however, for convex integration, the subsolutions are usually supposed to have extra regularity (at least continuity or piecewise continuity).  

In the second part, we consider several more involved conservation laws 
describing compressible fluids, 
and we prove the equivalence of energy-variational and dissipative weak solutions in each case.
Actually, the considered concept of dissipative weak solutions is stronger than the usual definition, where the constant by which the defect measure is estimated in terms of the energy defect is usually left undetermined. In contrast, we determine this constant exactly and give an additional characterization of the defect measure. 
For these refined dissipative weak solutions, we show equivalence to the correct energy-variational formulation. 
We exemplify this equivalence for the isentropic Euler system, the Euler--Korteweg system, and the Euler--Poisson system. 

The article is structured as follows. 
In Section~\ref{sec:prelim}
we introduce the basic notation 
and we prepare several auxiliary results
that are the fundament of later established equivalence results.
The incompressible Euler equations are treated in Section~\ref{sec:incomp.euler},
where we relate the solution concepts of
energy-variational, dissipative weak and measure-valued solutions
with each other.
In the remaining sections, we focus on the equivalence between energy-variational and dissipative weak solutions,
which we establish for the isentropic Euler, the Euler--Korteweg and the Euler--Poisson system in Sections~\ref{sec:compEul},~\ref{sec:eulerkorteweg} and~\ref{sec:eulerpoisson}, respectively.

\section{Notation and preliminaries}
\label{sec:prelim}
The Frobenius product of two matrices $A,B \in \R^{m \times d}$, $m,d\in\N$, is represented by $A:B:=A_{ij}B_{ij}$, where we implicitly sum over repeated indices.
By ${I}_d \in \R^{d\times d}$ we denote the identity matrix,
and $\mathbb{R}^{d\times d}_{\mathrm{sym}}$ denotes the space of real symmetric
matrices.
We write $(A)_{\text{sym}}=\frac{1}{2}(A+A^T)$ for the symmetric part of $A \in \R^{d\times d}$, and $(A)_{\text{sym},-}$, $(A)_{\text{sym},+}$, and  $(A)_{\text{sym},0}$ denote the negative semi-definite, positive semi-definite, and symmetric trace-less parts of the symmetric matrix $(A)_{\text{sym}}$, respectively. 
We further use $(a)_+=\max\{a,0\}$ and $(a)_-=\max\{-a,0\}$
to denote the positive and negative part of a scalar $a\in\R$. 

For Lebesgue and Sobolev spaces on a domain $\Omega\subset\R^d$, we use the standard notation $L^p(\Omega)$, $W^{k,p}(\Omega)$ and $H^k(\Omega)=W^{k,2}(\Omega)$ 
for $p\in[1,\infty]$ and $k\in\N$.
Here, we always consider the Lebesgue measure $\calL_\Omega$ on $\Omega$.
If $\mu$ is a different measure on $\Omega$, we write $L^p(\Omega;\mu)$
for the corresponding Lebesgue spaces.
Let $T>0$ and $\mathbb X$ be a Banach space.
For $p\in[1,\infty]$, Lebesgue--Bochner spaces are denoted by 
$L^p(0,T;\mathbb X)$. 
Moreover, the space of weakly* measurable functions $f:[0,T]\to \mathbb X^*$ that are essentially bounded  
is denoted by $L^\infty_{w^*}(0,T;\mathbb X^*)$.
Similarly, we write $L^\infty_{w^*}(M;\mathbb X^*)$
and $L^\infty_{w^*}(M;\mu;\mathbb X^*)$ if we consider the corresponding spaces 
on a subset $M\subset\R^d$ equipped with the Lebesgue measure or a different measure $\mu$, respectively.
The space ${C}_{w^*}([0,T];\mathbb X^*)$ consists of all weakly* continuous functions $f:[0,T]\to \mathbb X^*$.

If $M\subset\R^d$, $d\in\N$, is open or closed, we write $C(M)$ and $C^k(M)$
for the continuous and the $k$-times continuously differentiable functions on $M$ for $k\in\N\cup\{\infty\}$. 
The class of finite Radon measures $\mathcal M(M)$ on $M$ is
equipped with the norm given by the total variation
\[
\|\mu\|_{\mathcal M(M)}=\int_{M}\dd{|\mu|}=|\mu|(M),
\]
where the measure $|\mu|$ is the variation of $\mu\in\mathcal M(M)$.
By $\calP(M)\subset\calM(M)$ we denote the subset of probability measures.
For a compact set $ M\subset \R^d$, the dual space of $C(M)$ can be identified $\mathcal M(M)$
in terms of the duality pairing $\langle\mu,\varphi\rangle:=\int_M\varphi\,\dd\mu$.
In the case of the whole space $M=\R^d$ 
or the unit sphere $M=\calS^{d-1}\subset\R^d$
we use $\xi\in\R^d$ and $\theta\in\calS^{d-1}$ as respective dummy variables to express elementary functions. 
For instance, we write
\[
\langle \mu,|\xi|^2\rangle=\int_{\R^d} |\xi|^2\,\dd\mu(\xi),
\qquad
\langle \mu,\theta\rangle=\int_{\calS^{d-1}} \theta\,\dd\mu(\theta).
\]

We next recall some fundamental results on normal cones.
Let $H$ be a Hilbert space with inner product $(\cdot ,\cdot)$ and $\calC \subset H$ be a closed convex cone. We call 
$$ \calC^\circ := \{   z \in  H\mid \forall  y \in\calC: \ (z,y)\leq 0  \}$$
the       normal     cone of $\calC$. 
By $ P_{\calC}$ we denote the orthogonal projection in $H$ onto $\calC$. 
\begin{lem}\label{lemcone}
    Let $\calC\subset H$ be a closed convex cone. Then it holds
    \begin{enumerate}[label=(\roman*)]
         \item 
         \label{lemcan:1}
         $\calC=\calC ^{\circ\circ}$, that is,
         it holds $ y \in \calC$ if and only if 
         $(z,y) \leq 0$ for all $z \in \calC^\circ 
        $.
        \item $ I - P_\calC = P_{\calC^\circ}$, where $I: H \to H $ denotes the identity mapping.
        In particular, for all $y \in \calC$ and $z\in H$ it holds 
        $(z,y) \leq (P_{\calC}(z),y)$.
        \label{lemcan:5}
\end{enumerate}
\end{lem}

\begin{proof}
By definition of the normal cone, we always have
$\mathcal C \subset \mathcal C^{\circ\circ}$.
To show the reverse inclusion, let $y \notin \mathcal C$.
The Hahn--Banach separation theorem
yields the existence of a vector $z \in H$ such that
$(z,x) \le 0$ for all $x \in \mathcal C$,
and $(z,y) > 0$.
Therefore, $z \in \mathcal C^\circ$ and $y \notin \mathcal C^{\circ\circ}$.
In summary, this shows~\ref{lemcan:1}.



To show~\ref{lemcan:5}, we use Moreau's decomposition theorem~\cite{ConeMoreau62},
which implies that every $z \in H$ can be uniquely decomposed as
$z = P_{\mathcal C}(z) + P_{\mathcal C^\circ}(z)$
with $(P_{\mathcal C}(z), P_{\mathcal C^\circ}(z)) = 0$.
Moreover, for $y\in\calC$, we have 
$
(P_{\mathcal C^\circ}(z), y) \le 0
$ by~\ref{lemcan:1},
which yields
\[
(z,y) = (P_{\mathcal C}(z), y)+(P_{\mathcal C^\circ}(z), y)\leq (P_{\mathcal C}(z), y)
\]
and completes the proof.
\end{proof}

We use these general results on cones to characterize cone-valued Radon measures by duality.
Let $\Omega \subset \mathbb{R}^d$, $d\in\N$, be a bounded Lipschitz domain and let
$\mathcal{C} \subset \mathbb{R}^m$, $m\in\N$, be a closed convex cone.
Denote by $\mathcal{M}(\ov\Omega;\mathbb{R}^m)$ the space of finite
$\mathbb{R}^m$-valued Radon measures on $\ov\Omega$.
We define the subset of Radon measures with values in a cone $\calC\subset\R^{m}$, via 
\[
\calM (\ov\Omega ; \calC)
= \setcl{\mu\in \calM (\ov\Omega ; \R^{m})}{\forall A\subset\ov\Omega \text{ Borel measurable}:\ \mu(A)\in\calC}.
\]

\begin{lem}[Characterization of cone-valued Radon measures]
\label{lem:duality.conevalued}
Let $\mathfrak{R} \in \mathcal{M}(\ov\Omega;\mathbb{R}^m)$,
and let $\calC\subset\R^m$ be a closed convex cone.
Then $\mathfrak{R} \in \mathcal{M}(\ov\Omega;\mathcal{C})$
if and only if
\[
\int_{\ov\Omega} B \cdot d\mathfrak{R} \le 0
\quad \text{for all } B \in C(\ov\Omega;\mathcal{C}^\circ).
\]
\end{lem}

\begin{proof}
First assume that $\mathfrak{R} \in \mathcal{M}(\ov\Omega;\mathcal{C})$, and let
$B \in C(\ov\Omega;\mathcal{C}^\circ)$.
Using the Radon--Nikodým decomposition of $\mathfrak{R}$ with respect to its
total variation measure $|\mathfrak{R}|$, we may write
\[
    d\mathfrak{R}(x) = \theta(x)\, d|\mathfrak{R}|(x),
\]
where $\theta(x) \in \mathcal{C}$ for $|\mathfrak{R}|$-a.a.\ $x \in {\ov\Omega}$.
Hence,
\[
\int_{{\ov\Omega}} B \cdot d\mathfrak{R}
= \int_{{\ov\Omega}} B(x) \cdot \theta(x)\, d|\mathfrak{R}|(x)\leq 0
\]
since $B(x) \in \mathcal{C}^\circ$ and $\theta(x) \in \mathcal{C}$
implies $B(x) \cdot \theta(x) \le 0$ for $|\mathfrak{R}|$-a.a.\ $x \in {\ov\Omega}$.

Now assume that $\mathfrak{R} \notin \mathcal{M}({\ov\Omega};\mathcal{C})$.
Then there exists a Borel set $E \subset {\ov\Omega}$ such that
$\mathfrak{R}(E) \notin \mathcal{C}$. 
Since $\mathcal{C}$ is closed and convex, there
exists a vector $z \in \mathcal{C}^\circ$ such that
$
z \cdot \mathfrak{R}(E) > 0.
$
Since $z\cdot\mathfrak R$ is a scalar Radon measure and thus regular, 
there exist an open set $U$ and a compact set $K$
such that $K\subset E\subset U\subset\ov\Omega$
and 
$z \cdot \mathfrak{R}(U)\geq z \cdot \mathfrak{R}(E)\geq z \cdot \mathfrak{R}(K)>0$.
By Urysohn's lemma, there exists a continuous function 
 $\varphi \in C({\ov\Omega};[0,1])$ such that $\varphi \equiv  1 $ on $K$
 and $\varphi\equiv0$ in ${\ov\Omega} \setminus U$. 
Define $B(x) := \varphi(x)\, z$.
Then $B \in C({\ov\Omega};\mathcal{C}^\circ)$ and
\[
\int_{{\ov\Omega}} B \cdot d\mathfrak{R}
= \int_{{\ov\Omega}} \varphi \,z \cdot d\mathfrak{R}
\ge \int_{K} \varphi \,z \cdot d\mathfrak{R}
= z \cdot \mathfrak{R}(K) > 0.
\]

Combining both cases, 
we have shown the asserted equivalence.
\end{proof}

The existence of certain measures is later derived from the Riesz representation theorem and duality.
To  quantify the norm of the associated measures precisely,
one has to take into account the correct dual norm.
In particular, we will consider different norms on the space of symmetric matrices.

\begin{lem}[Matrix norms and duality for symmetric matrices]\label{lem:norm}
Denote the eigenvalues of $A\in\mathbb{R}^{d\times d}_{\mathrm{sym}}$ by
$\lambda_1(A),\dots,\lambda_d(A)\in\R$ and define the norms
\begin{align*}
|A|_2 &:= \max_{1\le i \le d} |\lambda_i(A)|,
&&\text{(spectral norm)}\\
|A|_{\mathrm{tr}} &:= \sum_{i=1}^d |\lambda_i(A)|,
&&\text{(trace/nuclear norm)}\\
|A|_{\mathrm F} &:= \sqrt{ A:A }
= \bigg(\sum_{i=1}^d \lambda_i(A)^2\bigg)^{1/2}.
&&\text{(Frobenius norm)}
\end{align*}
Then the dual norms with respect to the Frobenius product, 
defined by 
$A : B  = \mathrm{tr}(A^T B)$ for $A,B\in\mathbb{R}^{d\times d}_{\mathrm{sym}}$,
satisfy
\[
|A|_2^* = |A|_{\mathrm{tr}},\quad
|A|_{\mathrm{tr}}^* = |A|_2, \quad
|A|_{\mathrm F}^* = |A|_{\mathrm F}
\,.
\]
\end{lem}

\begin{proof}
All three norms are examples of Schatten norms defined on $\mathbb{R}^{d\times d}_{\mathrm{sym}}$,
and the asserted duality properties follow from elementary calculations.
For the correspondence in infinite-dimensional Hilbert spaces,
an overview can be found in~\cite[Appendix D]{hytoenen2016analysis} for example.
\end{proof}

The following lemma will be the central tool to
show sharp equivalence results 
between energy-variational and dissipative weak solutions.

\begin{lem}\label{lem:equi}
    Let $\mathbb Y \subset L^1(0,T;C^1({\ov\Omega} ; \R^m))$ 
    be a dense linear subspace.
    Let $\ell : \mathbb Y \to \R$ be a linear mapping,
    and let $\zeta \in L^\infty(0,T)$ with $\zeta\geq 0$ a.e.~in $(0,T)$. 
    Let $\calC _1,\,\calC_2\subset\R^{m\times d}$ be two cones, 
    and define $ \calL_i = \spa(\calC_i)$ for $i\in\{1,2\}$. 
    Then the following two assertions are equivalent: 
    \begin{enumerate}[label=(\roman*)]
        \item \label{item:1} 
        There exist $\frakR_1\in L^\infty_{w^*}(0,T;\mathcal{M}({\ov\Omega};\calC_1))$ and $\frakR_2\in L^\infty_{w^*}(0,T;\mathcal{M}({\ov\Omega};\calC_2))$ such that
        \begin{align}\label{eq:normfrakR}
        \begin{split}
            \langle\ell , \varphi\rangle =\int_0^T \bigg[\int_{{\ov\Omega}}
            \nabla \varphi
            : \dd \frakR_1 + \int_{{\ov\Omega}}
            \nabla \varphi
            : \dd\frakR_2 \bigg]\,\dd t \quad 
            &\text{for all }\varphi\in\mathbb Y,\\  \frac{1}{\eta_1}\| \frakR_1\|_{\mathcal{M}({\ov\Omega};\calL_1^*)} + \frac{1}{\eta_2}\| \frakR_2\|_{\calM({\ov\Omega};\calL_2^*)} \leq \zeta \quad 
            &\text{a.e.~in }(0,T) \,.
            \end{split}
        \end{align}
        \item \label{item:2}
        For all $\varphi\in \mathbb Y$  it holds 
        \begin{align}
            \label{eq:estlinear}
            \langle\ell, \varphi\rangle \leq \int_0^T\max\left \{ \eta_1 \| P_{\calC_1}(\nabla\varphi) \|_{L^\infty(\Omega;\calL_1)} , \eta_2 \| P_{\calC_2}(\nabla\varphi)\|_{L^\infty(\Omega;\calL_2)}\right\}\zeta\,\dd t. 
        \end{align}
    \end{enumerate}
    \end{lem} 
    \begin{remark}
    We note that the spaces $\calL_i$ are finite-dimensional, so that they can be identified with their dual spaces. Nevertheless, we use the notation $\calL_i^*$ in order to emphasize that $\calL_i^*$ is equipped
    with the corresponding dual norm.
    This makes the stated estimates rigorous without any additional absolute constants,
    which can arise by passing to a different norm.
    \end{remark}
\begin{proof}
    Let~\ref{item:1} be fulfilled. 
    From Lemma~\ref{lemcone} we conclude $ \int_{\ov\Omega} P_{\calC_i^\circ}(\nabla\varphi)\,\dd \frakR_i \leq 0 $ for $i=1,2$,
    so that we can estimate
    \begin{align*}
         \langle\ell , \varphi\rangle 
         \leq{}&\sum_{i=1}^2\int_0^T\int_{{\ov\Omega}}P_{\calC_i}(\nabla \varphi ) : \dd \frakR_i\dd t
         \\
         \leq{}& \sum_{i=1}^2 \int_0^T\eta_i\| P_{\calC_i}(\nabla \varphi ) \|_{L^\infty(\Omega;\calL_i )}\frac{1}{\eta_i}\| \frakR_i\|_{\calM({\ov\Omega};\calL_i^*)} \dd t 
         \\
         \leq{}& \int_0^T\max\left \{ \eta_1 \| P_{\calC_1}(\nabla\varphi) \|_{L^\infty(\Omega;\calL_1)} , \eta_2 \| P_{\calC_2}(\nabla\varphi)\|_{L^\infty(\Omega;\calL_2)}\right\}
         \zeta \,\dd t \,,
    \end{align*}
    which is exactly~\ref{item:2}. 

    Now let~\ref{item:2} be fulfilled. 
We  consider the product space $ X := C({\ov\Omega} ;\calL_1) \times C({\ov\Omega} ; \calL_2) $ equipped with the norm $ \| (\Psi_1,\Psi_2)\|_{X}:= \max_{i\in\{1,2\}}\{ \eta_i \| \Psi_i\|_{C({\ov\Omega};\calL_i)}\}$ 
and observe that the dual space may be identified with $ X^*= \calM({\ov\Omega};\calL_1^*)\times \calM({\ov\Omega};\calL_2^*)$ equipped with $ \| (\frakR_1 , \frakR_2) \|_{X^*} = 
\frac{1}{\eta_1}\| \frakR_1(t) \|_{\calM({\ov\Omega};\calL_1^*)}+ \frac{1}{\eta_2}\| \frakR_2(t) \|_{\calM({\ov\Omega};\calL_2^*)}$
using the usual duality pairing induced by the Riesz representation theorem.
On the space
\[
\begin{aligned}
\bar{X} := \big\{(\Psi_1,\Psi_2)\in  L^1(0,T;X) \mid \exists  \varphi \in \mathbb Y : \ & P_{\calL_i}(\nabla \varphi(x,t)) = \Psi_i(x,t) 
\\ 
& \text{ for all } i \in \{1,2\}, \,t\in(0,T),\,x\in\ov\Omega\big\}
\end{aligned}
\]
we define the linear mapping $ \tilde{\ell}: \bar{X} \to \R$ via 
$$
\langle \tilde{\ell} ; (\Psi_1,\Psi_2) \rangle = \langle \ell ,\varphi  \rangle  \,.
$$
First, we observe that this mapping is well defined. Indeed, let $\varphi \in \mathbb{Y}$ and $\psi  \in \mathbb{Y}$ be two elements such that $P_{\calL_i}(\nabla \varphi ) = \Psi_i=P_{\calL_i}(\nabla \psi )  $ in $(0,T)\times\ov\Omega$ 
for $i\in \{1,2\}$. Then it holds 
\begin{align*}
    \langle \ell , \varphi  \rangle - \langle \ell , \psi\rangle
    = 
    {}& \langle \ell , (\varphi-\psi)\rangle \\\leq{}& \max\left \{ \eta_1 \| P_{\calC_1}(\nabla(\varphi-\psi)) \|_{L^\infty(\Omega;\calL_1)} , \eta_2 \| P_{\calC_2}(\nabla(\varphi-\psi))\|_{L^\infty(\Omega;\calL_2)}\right\}\zeta 
    =0\,
\end{align*}
by~\eqref{eq:estlinear}.
Thus, the value $\langle\tilde \ell,(\Psi_1,\Psi_2)\rangle$
does not depend on the chosen representative $\varphi$, and $\tilde\ell$ is well defined on $\bar{X}$. 
Moreover, $\tilde\ell$ is linear.
Now, we may apply the Hahn--Banach theorem (for instance,~\cite[Thm.~1.1]{brezis2011fa})
in order to extend $ \tilde\ell$ to a linear mapping $\hat\ell : L^1(0,T;X) \to \R$ such that
\begin{align*}
    \langle \hat\ell , (\Psi_1,\Psi_2)  \rangle &= \langle \tilde{\ell} , (\Psi_1,\Psi_2) \rangle  \qquad \text{for all } (\Psi_1,\Psi_2)\in \bar{X}\,,\\
    \langle \hat\ell , (\Psi_1,\Psi_2) \rangle &\leq \int_0^T
    \max\left \{ \eta_1 \| P_{\calC_1}(\Psi_1) \|_{L^\infty(\Omega;\calL_1)} , \eta_2 \| P_{\calC_2}(\Psi_2)\|_{L^\infty(\Omega;\calL_2)}\right\}\zeta\,\dd t
    \\
    & \leq \int_0^T\| (\Psi_1,\Psi_2)\|_{X}\zeta \,\dd t  
    \qquad \text{for all } (\Psi_1,\Psi_2)\in L^1(0,T;X)\,.
\end{align*}
By the aforementioned duality due to the Riesz representation theorem, we obtain $ 
\frakR_i\in L^\infty_{w^*}(0,T;\calM({\ov\Omega};\calL_i^*))$ for $i\in\{1,2\}$ with 
\begin{align*}
     \int_0^T\bigg[\int_{{\ov\Omega}} \Psi_1 : \dd \frakR_1 + \int_{{\ov\Omega} } \Psi_2 : \dd \frakR_2\bigg]\dd t  =  \langle \hat\ell , (\Psi_1,\Psi_2)\rangle\leq \int_0^T\| (\Psi_1,\Psi_2)\|_{X}\zeta \,\dd t 
\end{align*}
for all $ (\Psi_1,\Psi_2) \in L^1(0,T;X)$.
From this, we conclude~\eqref{eq:normfrakR}. 
Moreover, we have
    \begin{align*}
        &\int_0^T\bigg[\int_{{\ov\Omega}} \Psi_1 : \dd \frakR_1 + \int_{{\ov\Omega} } \Psi_2 : \dd \frakR_2\bigg]\dd t 
        \\
        &\qquad
        \leq \int_0^T\max\left \{ \eta_1 \| P_{\calC_1}(\Psi_1) \|_{L^\infty(\Omega;\calL_1)} , \eta_2 \| P_{\calC_2}(\Psi_2)\|_{L^\infty(\Omega;\calL_2)}\right\}\zeta\dd t \,.
    \end{align*}
By choosing $(\Psi_1,\Psi_2)\equiv(A,0)$ for any $A\in C(\ov\Omega;\calC_1^\circ)$,
we observe that $\int_{{\ov\Omega}} A : \dd \frakR_1\leq 0 $.
Therefore, Lemma~\ref{lem:duality.conevalued} implies that 
$\frakR_1 \in L^\infty_{w^*}(0,T;\mathcal{M}({\ov\Omega} ; \calC_1))$,
and a similar argument shows that $ \frakR_2 \in L^\infty_{w^*}(0,T;\mathcal{M}({\ov\Omega} ; \calC_2))$.
We have thus shown~\ref{item:1}, which finishes the proof. 
\end{proof}

\section{The incompressible Euler equations}
\label{sec:incomp.euler}

Let $\Omega\subset\R^d$ be a bounded Lipschitz domain and $T>0$. 
The  Euler equations for incompressible fluids are given by
\begin{subequations}\label{eq:Incomp}
\begin{alignat}{2}
\partial_t \f v + ( \f v \cdot \nabla ) \f v + \nabla p = \f 0 , \quad \di \f v ={}& 0 \qquad && \text{in }\Omega \times (0,T)\,,
\\
\f v \cdot \f n = {}& 0 \qquad &&\text{on }\partial \Omega\times (0,T) \,,
\\
\f v (0) ={}& \f v_0 \qquad && \text{in } \Omega \,.
\end{alignat}
\end{subequations}
Here, $ \f v : \Omega\times (0,T) \to \R^d $ denotes the velocity of the fluid and $p : \Omega\times (0,T) \to \R$ denotes the pressure. 
We define the energy  $\E : \LRsigma{2}(\Omega) \to \R$ 
with $\E(\f v) := \frac{1}{2} \norm{\f v }_{L^2(\Omega)}^2$, where $\LRsigma{2}(\Omega)$ denotes the closure of the set
$\setcl{\varphi\in C^\infty_c(\Omega;\R^d)}{\dv\varphi=0}$
in $L^2(\Omega;\R^d)$. 

We introduce the concepts of energy-variational, dissipative weak and measure-valued solutions
to~\eqref{eq:Incomp} as follows.

\begin{dfn}\label{def:envarIncommp}
A pair $(\f v , E )\in  L^\infty(0,T; L^2_\sigma ( \Omega ))  \times \BV([0,T];\R) $  is called an \textit{energy-variational solution} to the incompressible Euler equations~\eqref{eq:Incomp}
if $E(t) \geq \E (\f v(t) )$ for a.e.~$t\in (0,T)$
and if the inequality
\begin{equation}
\left [ E - \int_{\Omega} 
 \f v  \cdot \f \varphi 
\,\dd  \f x  \right ] \bigg|_s^t 
+ \int_s^t \int_{\Omega}
\f v \cdot \partial_t \f \varphi 
+ { \f v  \otimes \f v}  : \nabla \f \varphi 
\,\dd  \f x 
+ \mathcal{K}(\f \varphi ) \left [ \E(\f v) - E\right ] 
\dd \tau
\leq 0
\,\label{relenIncomp}
\end{equation}
holds for all test functions $\f \varphi\in C^1(\ov \Omega \times [0,T]; \R^{d})$ with $ \di \f \varphi = 0$ in $\Omega \times [0,T]$, and $ \f \varphi \cdot \f n = 0 $ on $ \partial \Omega\times [0,T]$
and for a.a.~$s,t\in(0,T) $, $s<t$, including $s=0$ with $\f v (0)  =  \f v_0 $,
where
\begin{equation}\label{eq:regweight.incomp}
\mathcal{K}(\f \varphi) =2
\|(\nabla \f \varphi )_{\sym,-}\| _{L^\infty(\O;\R^{d\times d})}\,.
\end{equation}
\end{dfn}

 \begin{dfn}\label{def:disssol.incomp}
A pair   $(\f v , E )\in  L^\infty(0,T; L^2_\sigma ( \Omega ))  \times \BV([0,T];\R) $ is called
a \textit{dissipative weak solution} to the incompressible Euler equations~\eqref{eq:Incomp} if there exists a \textit{Reynolds defect}  $ \frakR \in  L^\infty _{w^*} (0,T;\mathcal{M}(\ov \O ; \mathbb{R}^{d\times d}_{\sym,+}))$ such that the equation 
\begin{align}
\int_{\Omega} \f v \cdot  \f \varphi\,\dd  \f x \Big|_s^t
=  \int_s^t\!\! \int_{\Omega} \f v \cdot \partial_t \f \varphi  +  \f v \otimes \f v:  \nabla \f \varphi\,\dd  \f x\dd  \tau   + \int_s^t\!\!\int_{\Omega}   \nabla \f \varphi:\dd   \frakR_\tau \dd  \tau \label{measeq}
\end{align}
holds for all $\f \varphi\in C^1(\ov \Omega \times [0,T]; \R^{d})$ with $ \di \f \varphi = 0$ in $\Omega \times [0,T]$ and $ \f \varphi \cdot \f n = 0 $ on $ \partial \Omega\times [0,T]$,
and for  a.a.~$s,t\in(0,T) $, including $s=0$ with $\f v(0)=\f v_0$,
and if $E$ is a   non-increasing function 
such that 
\begin{equation}
\E(\f v (t))+ \frac{1}{2}\int_{\Omega} I :  \dd  \frakR_t   \leq E(t)   \,\label{diseneq}
\end{equation}
for a.a.~$t\in(0,T)$.
\end{dfn}
 \begin{dfn}\label{def:meas.incomp}
We call a triple   $(\nu^0,\lambda , \nu^\infty )$ 
a \textit{measure-valued solution} to the incompressible Euler equations~\eqref{eq:Incomp} if 
\begin{align*}
\nu^o &\in L_{w^*}^\infty(\Omega\times(0,T);\mathcal{P}(\R^d)),\\
\lambda &\in L^\infty_{w^*}(0,T; \calM^+(\ov{\Omega}) ), \\
\nu^\infty &\in L^\infty_{w^*}(\ov{\Omega}\times(0,T) ;   \lambda_t \otimes \mathcal{L}_{(0,T)}; \calP(\calS^{d-1})  )),
\end{align*}
if 
the equation 
\begin{equation}
\begin{aligned}
  \int_{\Omega} \langle \nu^o, \xi\rangle \cdot  \f \varphi\,\dd  \f x \Big|_s^t =&\int_s^t\!\! \int_{\Omega}\langle \nu^o, \xi\rangle  \cdot \partial_t \f \varphi  +  \langle \nu^o , \xi \otimes \xi\rangle :  \nabla \f \varphi\,\dd  \f x\dd  \tau  \\
  &\quad+\int_s^t\!\!\int_{\ov\Omega}   \nabla \f \varphi: \langle \nu^\infty, \theta\otimes \theta \rangle \,\dd \lambda \dd  \tau  \label{measeq2}
\end{aligned}
\end{equation}
holds for all for all test functions $\f \varphi\in C^1(\ov \Omega \times [0,T]; \R^{d})$ with $ \di \f \varphi = 0$ in $ \Omega \times [0,T]$ and  $ \f \varphi \cdot \f n = 0 $ on $ \partial \Omega\times [0,T]$, 
and for  a.a.~$s,t\in(0,T) $, $s<t$, 
and if
 \begin{equation}
    \langle \nu^o_{\cdot,t},\xi\rangle \in L^2_{\sigma}(\Omega)  \label{eq:measdiv}
\end{equation}
for  a.e.~$t\in(0,T)$,
and if it holds
\begin{equation}
 \frac{1}{2}\int_\Omega \langle \nu^o_{x,t}, |\xi|^2 \rangle\,\dd x + \frac{1}{2}\lambda_t(\ov\Omega)  \leq   \frac{1}{2}\int_\Omega \langle \nu^o_{x,s}, |\xi|^2 \rangle \dd x + \frac{1}{2}\lambda_s(\ov\Omega)  \,\label{measeneq}
\end{equation}
for a.a.~$s,t\in(0,T) $, $s<t$. 
The initial value is attained via $  \langle \nu^o_{x,0}, \xi\rangle = \f v_0(x) $, 
where the additional regularity $ t \mapsto \int_\Omega \langle \nu^o_{x,t} , \xi \rangle \cdot \varphi (x) \,\dd x \in C([0,T])$ for all $\varphi \in C(\ov\Omega)$  is used, which is a consequence of~\eqref{measeq2}. 
\end{dfn}
In fact, the weak continuity in time of the velocity field, that is, the continuity of $t\mapsto\int_\Omega \f v(t)\cdot \f \varphi(t)\, \dd x$ for all $\varphi\in L^2_\sigma(\Omega)$,
is also true for energy-variational and dissipative solutions, as follows from~\cite[Lemma 8]{delellis}.

\begin{remark}
\label{rem:en.ineq}
    All three solution concepts include a global energy inequality in a strong form,
    that is, an inequality for the total energy at different times $s<t$.
    For energy-variational solutions and dissipative weak solutions, 
    this energy inequality corresponds to the property that $E$ is nonincreasing,
    while for measure-valued solution, this is encoded in the inequality~\eqref{measeneq},
    see also Theorem~\ref{thm:main} below.
    There are other solution concepts
    that feature an energy inequality that only compares with the initial energy at time $s=0$,
    or that consider an inequality for the local energy evolution.
    This leads to different types of solutions, which are not equivalent to those considered here. 
\end{remark}

\begin{thm}\label{thm:main}
    Let $\Omega$ be a bounded Lipschitz domain. Then the following statements are equivalent:
    \begin{enumerate}[label=(\roman*)]
        \item   $(\f v ,E)$ is an energy-variational solution in the sense of Definition~\ref{def:envarIncommp}. \label{pr:envar}
        \item  $(\f v , E)$  is  a dissipative weak solution in the sense of Definition~\ref{def:disssol.incomp}.\label{pr:diss}
        \item  There exists a measure-valued solution $(\nu^o, \lambda , \nu^\infty)$ in the sense of Definition~\ref{def:meas.incomp} 
        with $ \langle \nu ^o, \xi\rangle = \f v $ and $ E = \frac{1}{2}\int_{\Omega} \langle \nu^o , |\xi|^2 \rangle \de x + \frac 1 2 \lambda_t({\ov\Omega})
        $.\label{pr:meas}
    \end{enumerate}
\end{thm}

\begin{proof}
The first two authors proved the equivalence of~\ref{pr:envar} and~\ref{pr:diss} already in the case of a torus in~\cite{eiterlasarzik2024envar}, but we repeat the argument here for the readers' convenience. 

We first show that \ref{pr:envar} implies \ref{pr:diss}.
    Let~$(\f v ,E)$ be an energy-variational solution in the sense of Definition~\ref{def:envarIncommp}. 
  The choice $ \f \varphi = \f 0$ implies that $ E $ is non-increasing.
 For $ \f \psi \in C^1(\O\times[0,T];\R^d)$ with $\dv\f\psi=0$, we may choose $\f \varphi =\alpha  \f \psi$ for $\alpha >0$. Multiplying the resulting inequality~\eqref{relenIncomp} by $ \alpha^{-1}$, we  infer
\begin{multline}
\left [ \frac{1}{\alpha} E - \int_{\O} 
\f v  \cdot \f \psi 
\,\dd\f x  \right ] \bigg|_{s-}^t 
+ \int_s^t \int_{\O}
\f v \cdot \partial_ t \f \psi 
+ { \f v  \otimes \f v}  : \nabla \f \psi 
\,\dd\f x 
+ \mathcal{K}(\f \psi ) \left [ \E(\f v) - E\right ] 
\dd s
\leq 0\,
\end{multline}
for  almost all $s<t \in [0,T]$ including $s=0$. In the limit $\alpha\to \infty$, and setting $s=0$, $t=T$ (which is allowed thanks to the continuity of $t\mapsto \int_{\O} 
\f v  \cdot \f \psi 
\,\dd\f x$), we find 
\begin{align}
- \int_{\O} \f v \cdot \f \psi  \,\dd\f x   \Big|_{0}^T + \int_0^T \int_{\O} \f v \cdot \partial_ t \f\psi + ( \f v \otimes \f v ) : \nabla \f \psi \,\dd\f x \dd t \leq \int_0^T \mathcal{K}(\f \psi ) [ E - \E(\f v) ]\, \dd t \label{estincomp}
\end{align}
for all $ \f \psi \in C^1(\O\times[0,T];\R^d)$ with $\dv\f\psi=0$.   
We define
\[
\begin{aligned}
&\mathcal V:= \{\nabla \f \varphi \mid \f \varphi \in C^1(\overline\O\times[0,T];\R^d) ,\, \di \f \varphi = 0 \text{ in }\O \times (0,T)
 \,,\   \f \varphi \cdot \f n  =0 \text{ on }\partial \Omega
\}\,,
\\
&\ell \colon\mathcal V\to\R,
\quad
\langle \ell, \nabla \f \psi \rangle := - \int_{\O} \f v \cdot \f \psi  \,\dd\f x   \Big|_{0}^T + \int_0^T \int_{\O} \f v \cdot \partial_ t \f\psi + ( \f v \otimes \f v ) : \nabla \f \psi \,\dd\f x \dd t\,,
\\
&\mathfrak p\colon L^1(0,T;\calC(\overline\O ; \R^{d\times d}))  \to \R,
\quad
\mathfrak p(\Phi) := \int_0^T 2\| (\Phi)_{\sym,-}\|_{\calC_0(\overline\O;\R^{d\times d})}(E - \mathcal{E}(\f v))\dd t\,
\end{aligned}
\]
Noting that both sides of~\eqref{estincomp} and the definition of $\ell$ 
are invariant under addition of a constant to $\psi$, 
we observe that~\eqref{estincomp} yields
$
\langle \ell , \Phi\rangle \leq \mathfrak p( \Phi) \,
$
for all $\Phi\in\calV$.
Since $\mathfrak p$ is a sublinear functional, 
by a combination of the Hahn--Banach theorem~\cite[Thm.~1.1]{brezis2011fa} and the Riesz representation theorem,
we infer the existence of an element
$\mathfrak R \in L^\infty_{w^*}(0,T;\mathcal{M}(\overline\O ; \R_{\sym}^{d\times d}
))$ 
with
\[
\begin{aligned}
\forall\Phi\in L^1(0,T; \calC(\overline\O; \R_{\sym}^{d\times d })): &\qquad \langle -\mathfrak R, \Phi\rangle \leq \mathfrak p(\Phi),
\\
\forall\Phi=\nabla\psi \in  \mathcal{V} : &\qquad \langle -\mathfrak R, \nabla \f \psi \rangle = \langle \ell ,\nabla  \f \psi \rangle.
\end{aligned}
\]
The first property implies $\langle \mathfrak R, \Phi\rangle \geq 0$ 
if $\Phi(x,t)$ is positive semi-definite for a.a.~$(x,t)\in\O\times (0,T)$, 
so that we have $\mathfrak R \in L^\infty_{w^*}(0,T;\mathcal{M}(\overline\O ; \R_{\sym,+}^{d\times d}))$. 
The second property yields \eqref{measeq} 
for $\f \nabla\psi\in \mathcal V$.
Considering
$\Phi(x,t)= - \phi(t)I_d$ for some $\phi\in\calC_0^{1}([0,T))$ with $\phi \geq 0$,
we further have
\[
\int_0^T \phi(t) \int_{\O} I :  \dd\frakR(t) \dd t
=\langle -\mathfrak R, \Phi\rangle 
\leq \mathfrak p(\Phi)
=2\int_0^T \phi(t) (E - \mathcal{E}(\f v))\,\dd t.
\]
Since $\phi\geq 0$ is arbitrary,  this directly implies \eqref{diseneq} 
for a.e.~$t\in(0,T)$.
In total, we see that $(\f v, E)$ is a dissipative weak solution.

To show that \ref{pr:diss} implies \ref{pr:meas},
let $(\f v, E)$ be a dissipative weak solution in the sense of Definition~\ref{def:disssol.incomp} with Reynolds defect $ \frakR $.
We alter the trace of the Reynolds defect via 
\[
\dd\tilde{\frakR}:=\dd\frakR + \frac1{|\Omega|d} \left [ 2 (E- \E(\f v)) -\tr{\frakR}(\ov\Omega)\right ]I_d \,\dd x.
\]
Then the matrix-valued measure $\tilde{\frakR}$ remains positive semi-definite
due to 
condition~\eqref{diseneq}. 
As the functions $\f\varphi$ in~\eqref{measeq} are divergence free, this equation remains valid for $\frakR$ replaced with 
$\tilde{\frakR}$. 
The Radon--Nikod\'ym--Lebesgue decomposition yields a function
$\frakR^o \in L^1(\Omega; \R^{d\times d}_{\sym,+})$ and a singular measure $\lambda \in \mathcal{M}({\ov\Omega};\R_+)$ concentrated on sets of  Lebesgue measure zero, together with $\frakR^\infty \in L^1({\ov\Omega}, \lambda ; \R^{d\times d}_{\sym,+})$ satisfying $\tr \frakR^\infty_{x,t} = 1 $, such that 
\[
\dd \tilde{\frakR} = \frakR^o\dd x  + \frakR^\infty \dd {\lambda}.
\]
Using this decomposition, we 
define the oscillation measure $ \nu^o_{x,t}$ as a probability measure on $ \R^d$ with mean $\f v(x,t)$ and covariance $\frakR^o_{x,t} $, and the concentration-angle measure $\nu^\infty_{x,t}$ as a probability measure on $ \calS^{d-1}$ with mean $\f 0$ and variance $\frakR^\infty_{x,t} $. 
There are many possible choices to achieve this. 
We may choose $ \nu^o_{x,t}\in \calP(\R^d)$ to be a Gaussian measure  $\nu^o_{x,t} =   \mathcal{N}(\f v(x,t), \frakR^o_{x,t}) $, which, in the case of a strictly positive definite covariance matrix $ \frakR^o_{x,t}$, can be written via 
\[
\dd\nu^o_{x,t} :=\frac{1}{(2 \pi)^{d/2} \det(\frakR^o_{x,t})^{1/2}} \exp\left( -\frac{1}{2} (y - \f v (x,t))^{\top} (\frakR^o_{x,t})^{-1} (y - \f v (x,t)) \right)\,\dd y.
\]
This representation degenerates in certain directions 
if  $ \frakR^o_{x,t}$ is not strictly positive definite. 
In particular,
\EEE
if $ \frakR^o_{x,t}=0$, we simply set $\nu_{x,t}^o=\delta_{v(x,t)}$.
Since the matrix $\frakR^\infty_{x,t} $ is also positive semi-definite
for a.a.~$(x,t)\in\Omega\times(0,T)$, 
we may consider its eigendecomposition such  that $$\frakR^\infty_{x,t}
  = \sum_{i=1}^d \sigma^i_{x,t}\f e^i_{x,t} \otimes \f e^i_{x,t} \,,$$ where $ \sigma^i_{x,t}\geq 0$ and $ \f e^i_{x,t}$, $i=1,\dots,d$,
  denote the eigenvalues and eigenvectors of $ \frakR^\infty_{x,t}$, respectively. We define 
\[
\nu^\infty_{x,t} := \sum_{i=1}^d \frac{\sigma^i_{x,t}}{2} \big(\delta_{\f e^i_{x,t}}+\delta_{-\f e^i_{x,t}}\big),
\]
which is a probability measure on $\calS^{d-1}$ since 
\[
\int_{\calS^{d-1}}\,\dd \nu^\infty_{x,t}  =\sum_{i=1}^d \sigma^i_{x,t} = \tr {\frakR^\infty_{x,t}}  = 1 \,.
\]
By definition, the expectation of $\nu^\infty_{x,t}$ vanishes and 
its second moment is given by $\frakR^\infty_{x,t} $. 
In particular, this implies 
\[
    \dd \tilde{\frakR}_{x,t} = \frakR^o_{x,t}\dd x  + \frakR^\infty_{x,t}\dd \lambda= \langle \nu^o_{x,t}, ( \xi - \f v (x,t))\otimes ( \xi - \f v (x,t))\rangle \,\dd x + \langle \nu^\infty_{x,t}, \theta \otimes \theta \rangle \,\dd \lambda \,,
\]
where 
\begin{align*}
    \langle \nu^o
    , ( \xi - \f v 
    )\otimes ( \xi - \f v
    )\rangle 
     &= 
    \langle \nu^o
    ,  \xi \otimes  \xi \rangle  - \langle \nu^o
    , \xi \rangle \otimes \f v 
    - \f v 
    \otimes \langle \nu^o
    , \xi \rangle + \langle \nu^o
    , 1 \rangle \f v 
    \otimes \f v
    \\&
     = 
    \langle \nu^o
    ,  \xi \otimes  \xi \rangle  - \f v 
    \otimes \f v
    \,.
\end{align*}
With these relations, we can rewrite~\eqref{measeq} as 
\begin{align}\label{measeqeq}
\begin{split}
        - \int_\Omega \langle \nu ^o, \xi \rangle \cdot \f \psi \,\dd x \Big |_s^t + \int_s^t\int_{\Omega}\langle \nu^o, \xi\rangle  \cdot \partial_t \f \psi  +  \langle \nu^o , \xi \otimes \xi\rangle :  \nabla \f \psi\,\dd  \f x\dd  \tau  \\ + \int_s^t\!\!\int_{{\ov\Omega}}   \nabla \f \psi: \langle \nu^\infty, \theta\otimes \theta \rangle \,\dd \lambda \dd  \tau &= 0\,, \end{split}  
\end{align}
for all $ \f \psi \in C^1({\ov\Omega}\times[0,T];\R^d )$ with $\dv\f\psi=0$, which is equivalent to~\eqref{measeq2}. 
Moreover, from the definition of~$\tilde{\frakR}$, we get
\begin{align*}
   2[ E(t)- \E( \f v(t))  ] &=  \int_{\ov\Omega} I: \dd \tilde{\frakR}_t = \int_{\Omega} I : \frakR^o_{x,t}\,\dd x + \int_{\ov\Omega} I : \frakR^\infty_{x,t}\,\dd \lambda_t(x) \\
   &= \int_{\Omega} \langle \nu^o_{x,t} , |\xi-\f v(x,t)|^2 \rangle \,\dd x + \int_{\overline\Omega}\,\dd\lambda_t\\
   & = 
 \int_\Omega \langle \nu^o_{x,t}, |\xi|^2 \rangle \dd x - 2\E( \f v) + \lambda_t({\ov\Omega})\,,
\end{align*}
which yields the claimed identity for $E$.
Now~\eqref{measeneq} follows from monotonicity of $E$.

Finally, we show that \ref{pr:meas} implies \ref{pr:envar}.
Let $ ( \nu^o, \lambda, \nu^\infty)$ be a measure-valued solution in the sense of Definition~\ref{def:meas.incomp}. We define 
\[
    \f v (x,t):= \langle \nu^o_{x,t} , \xi \rangle \quad \text {and} \quad E (t) : =  \frac{1}{2}\int_\Omega \langle \nu^o_{x,t}, |\xi|^2 \rangle \dd x + \frac{1}{2}\lambda_t(\overline\Omega).  
\]
By Jensen's inequality, we find 
\[
\begin{aligned}
    \E( \f v) &= \frac{1}{2}\int_\Omega | \f v |^2 \,\dd x = \frac{1}{2}\int_\Omega | \langle \nu^o_{x,t} , \xi \rangle |^2 \,\dd x \leq \frac{1}{2}\int_\Omega  \langle \nu^o_{x,t} , |\xi|^2 \rangle   \,\dd x 
    \\ &\leq  \frac{1}{2}\int_\Omega \langle \nu^o_{x,t}, |\xi|^2 \rangle \,\dd x + \frac{1}{2}\lambda_t({\ov\Omega})  = E(t)\,,
\end{aligned}
\]
where we used that $\lambda_t$ is a nonnegative measure.
Moreover, together with~\eqref{eq:measdiv},
this implies $\f v \in L^\infty(0,T;L^2_{\sigma}(\Omega))$. 
To obtain~\eqref{relenIncomp}, let $s<t$ 
and choose $\varphi = \phi^n \f \psi $ in~\eqref{measeq2},
where $ \{ \phi^n\} \subset \calC_c^\infty [0,T)$ approximates the indicator function $ \chi_{[s,t]}$, and $\f \psi \in C^1 ( {\ov\Omega} \times [0,T]; \R^d)$ with $\dv\f\psi=0$, we find that~\eqref{measeq2} implies~\eqref{measeqeq}.
By the definition of $\f v $, we obtain
\begin{align}\label{eq:meastoenvar}
\begin{split}
     - &\int_\Omega \f v  \cdot \f \psi \,\dd x \Big |_s^t + \int_s^t\int_{\Omega}\f v  \cdot \partial_t \f \psi  +  \f v \otimes \f v  :  \nabla \f \psi\,\dd  \f x\dd  \tau  \\ &+ \int_s^t\!\!\int_{\Omega}   \nabla \f \psi: \big(\langle \nu^o , \xi \otimes \xi\rangle - \f v \otimes \f v \big) \,\dd x + \int_{{\ov\Omega}}   \nabla \f \psi:  \langle \nu^\infty, \theta\otimes \theta \rangle \,\dd \lambda \dd  \tau = 0\,
\end{split}
\end{align}
for all $ \f \psi \in C^1({\ov\Omega}\times[0,T];\R^d )$ with $\dv\f\psi=0$.
To estimate the second line from below,
we use the elementary identity
$ \f A : \f B = \f A : (\f B)_{\sym}$ for all $ \f A \in \R^{d\times d}_{\sym}$, $\f B\in\R^{d\times d}$
and the inequality $ \f A : \f B \geq  (\f A)_- : \f B$ for all $ \f A \in \R^{d\times d}_{\sym}$, $\f B\in\R^{d\times d}_{\sym,+}$. 
We observe 
\begin{align*}
    &\int_{\Omega}   \nabla \f \psi: \left(\langle \nu^o , \xi \otimes \xi\rangle - \f v \otimes \f v \right) \dd x + \int_{\Omega}   \nabla \f \psi:  \langle \nu^\infty, \theta\otimes \theta \rangle\, \dd \lambda
    \\
    &=     \int_{\Omega}   (\nabla \f \psi)_{\sym}: \left(\langle \nu^o , \xi \otimes \xi\rangle - \f v \otimes \f v \right) \dd x + \int_{{\ov\Omega}}    (\nabla \f \psi)_{\sym}:  \langle \nu^\infty, \theta\otimes \theta \rangle \,\dd \lambda
    \\
    &\geq \int_{\Omega}   (\nabla \f \psi)_{\sym,-}: \left(\langle \nu^o , \xi \otimes \xi\rangle - \f v \otimes \f v \right) \dd x + \int_{{\ov\Omega}}    (\nabla \f \psi)_{\sym,-}:  \langle \nu^\infty, \theta\otimes \theta \rangle\, \dd \lambda
    \\
    &\geq -\int_{\Omega}   \big|(\nabla \f \psi)_{\sym,-}\big| \tr\left(\langle \nu^o , \xi \otimes \xi\rangle - \f v \otimes \f v \right) \dd x- \int_{{\ov\Omega}}    |(\nabla \f \psi)_{\sym,-}|\tr(  \langle \nu^\infty, \theta\otimes \theta \rangle)\,  \dd \lambda 
    \\
    & \geq - 2 \|(\nabla \f \psi)_{\sym,-} \|_{\infty} \frac{1}{2} \left(\int_{\Omega}    \langle \nu^o , |\xi|^2\rangle \dd x - \int_\Omega|\f v|^2  \dd x  
    +\int_{{\ov\Omega}}    \langle \nu^\infty, 1\rangle\,  \dd \lambda \right)
    \\
    & = 2 \|(\nabla \f \psi)_{\sym,-} \|_{\infty} \left[ \E( \f v) - E\right]\,,
\end{align*}
For the second inequality, we use the duality between the Frobenius norm and the trace norm on the symmetric matrices,
compare Lemma~\ref{lem:norm}. 
The last equality is due 
$\nu^\infty_{x,t}\in\calP(\ov\Omega)$ and the definition of $E$.
Combining this inequality with~\eqref{eq:meastoenvar} and adding $ E \big|_s^t\leq 0$, which holds due to~\eqref{measeneq}, 
we obtain~\eqref{relenIncomp}
with $\mathcal K$ given by~\eqref{eq:regweight.incomp}.
Therefore, $ ( \f v ,E)$ is an energy-variational solution in the sense of Definition~\ref{def:envarIncommp}.
\end{proof}

We make the following observation, which is novel on bounded domains (on $\R^3$, the existence of measure-valued solutions was already shown in~\cite{DiPernaMajda}): 
\begin{cor}[Existence of measure-valued solutions]
    Let $\Omega\subset\R^d$ be a bounded domain with Lipschitz boundary. Then, for every $v_0\in L^2_\sigma(\Omega)$ and $T\in(0,\infty]$, there exists a measure-valued solution of the incompressible Euler system in the sense of Definition~\ref{def:meas.incomp}.
\end{cor}
\begin{proof}
    This follows from Theorem~\ref{thm:main} and the existence of energy-variational solutions, which can be shown as in~\cite{eiterlasarzik2024envar, Marcel}.
\end{proof}
While there might be more direct proofs that avoid the detour via energy-variational solutions, we note that the standard construction as a limit of solutions to the Navier--Stokes equations with Dirichlet boundary condition would not work in a bounded domain due to the possible formation and detachment of a boundary layer. We also remark that similar existence statements follow for all the examples discussed in this paper (compressible Euler, Euler--Poisson, Euler--Korteweg). 

\begin{remark}
We note that there are possibly uncountably many different measure-valued solutions we could define for one particular dissipative weak solution. Indeed, as the measure-valued solution only prescribes the expectation and the variance of the oscillation measure $\nu^o$ and the concentration-angle measure $\nu^\infty$, there are multiple ways to construct these,
while giving rise to the same dissipative weak solution as they have the same  expectation and variance. Nevertheless, the energy inequality turns out strong enough to exclude such non-uniqueness in case there is a strong solution emanating from the given data~\cite{weakstrongeuler}. This is the \emph{weak-strong uniqueness principle} that holds for measure-valued solutions of many fluid systems~\cite{wiedemannsurvey}. 
\end{remark}

\section{Isentropic Euler equations in a bounded domain}\label{sec:compEul}

In this section, we consider the compressible isentropic Euler equations in a bounded domain $\Omega \subset \R^d$ with Lipschitz boundary, which is given by 
\begin{subequations}\label{eq:isenEul}
    \begin{align}
\partial_t \rho + \text{div} \; m &=0 \quad &&\text{in}\; \Omega \times (0,T) ,\\
   \partial_t m + \text{div}\left( \frac{m \otimes  m}{\rho}\right) + \nabla( \rho^\gamma)&= 0\quad &&\text{in} \;\Omega \times (0,T),  \\
   (\rho(0),m(0))&=(\rho_0,m_0) \quad &&\text{in} \;\Omega,
\end{align}
with an adiabatic exponent $\gamma > 1$. The system is equipped with an impermeability boundary condition, that is,
\begin{align}
  m \cdot n&=0{} \qquad \text{on } \partial \Omega\times(0,T).
\end{align}\end{subequations}

In the case of periodic boundary conditions, 
that is, 
for the compressible Euler equations posed on the torus,
the existence of energy-variational solutions 
and their equivalence to dissipative weak solutions
was already shown in~\cite[Sec.~5]{eiterlasarzik2024envar}.
An analogous existence result on the multidimensional whole space
with nontrivial spatial limits 
was recently derived in~\cite{Schindler}. 

To introduce energy-variational solutions in the present case, 
we define the energy $\E:L^{\gamma}(\Omega) \times L^{\frac{2\gamma}{\gamma+1}}(\Omega;\R^d) \rightarrow [0,\infty]$ by
$$\mathcal{E}(\rho,m)=\int_\Omega \eta(\rho(x),m(x))\,\dd{x}$$ 
for the convex function $\eta: \mathbb{R}^{d+1}\rightarrow [0, \infty]$ given by
\begin{align}\label{eq:defeta}
    \eta(\rho, m) = \begin{cases}
    \frac{|m|^2}{2\rho}+\frac{\rho^\gamma}{\gamma-1} \quad &\text{if} \quad \rho>0, \\
    0 \quad &\text{if}\quad (\rho,m)=(0,0),\\
    \infty \quad &\text{else}.
       \end{cases}
\end{align}
As in~\cite{eiterlasarzik2024envar,Schindler},
a suitable regularity weight $\mathcal{K}:C^1({\ov\Omega} ;\R^d)\times \mathcal{V} \to [0,\infty)$ is given by an effective regularity weight
$\tilde{\mathcal{K} }:\mathcal{V} \to [0,\infty)$ via 
\begin{align}\label{eq:regweight.Euler.compr}
      \mathcal{K}(\psi,\varphi)=\tilde{\mathcal{K} } (\varphi) =\max\{ 2\|(\nabla \varphi)_{\mathrm{sym},-}\|_{L^\infty(\Omega;\R^{d\times d})},(\gamma-1)\|(\di\varphi)_{-}\|_{L^\infty(\Omega)}\}\,,
\end{align}
where $\mathcal V$ is given by $\mathcal{V}=\{\varphi \in C^1(\overline{\Omega};\mathbb{R}^d)\mid n\,\cdot \varphi=0\text{ on } \partial\Omega\}$. 

\begin{dfn}\label{def:envarEuler}
A triple 
\[
(\rho,m,E)\in L^\infty(0,T; L^\gamma(\Omega)) \times L^\infty(0,T;L^\frac{2\gamma}{\gamma +1}(\Omega;\R^d)) \times \mathrm{BV}([0,T])
\]
is called an \textit{energy-variational solution} to the isentropic Euler system~\eqref{eq:isenEul} if it holds $\mathcal{E}(\rho(t),m(t)) \leq E(t) $ for a.e.~$t \in (0,T)$ and  if the energy-variational inequality
\begin{align}\label{eq:envarEul}
\begin{split}
    \left[E- \int_\Omega \rho  \psi + m \cdot\,\varphi \,\dd{x} \right]\bigg|_s^t &+\int_s^t \int_\Omega \rho \, \partial_t \psi+ m \cdot\nabla \psi  \,\dd{x} \dd{\tau}\\&+ \int_s^t \int_\Omega m \cdot \,\partial_t \varphi + \left(\frac{m \otimes m}{\rho}+ \rho^\gamma {I}_d\right):\nabla \varphi \,\dd{x} \dd{\tau} \\
     &+ \int_s^t \tilde{\mathcal{K} }(\varphi)(\mathcal{E}(\rho,m)-E) \,\dd{\tau}\leq 0\end{split}
\end{align}
holds for all test functions $(\psi, \varphi) \in C^1([0,T];C^1({\ov\Omega}) \times \mathcal{V}) $  and for a.a.~$s,t\in (0,T)$, $s<t$, including $s=0$ with $(\rho(0), m(0)) = (\rho_0,m_0)$.

\end{dfn}

\begin{dfn}\label{def:dissweakEuler}
A triple 
\[
(\rho,m,E)\in L^\infty(0,T; L^\gamma(\Omega)) \times L^\infty(0,T;L^\frac{2\gamma}{\gamma +1}(\Omega;\R^d)) \times \mathrm{BV}([0,T])
\]
is called a \textit{dissipative weak solution}  to the isentropic Euler system~\eqref{eq:isenEul} if $ E $ is nonincreasing and there exists a Reynolds stress $\frakR_1 +\frakR_2I_d   $ consisting of an advective part $\frakR_1 \in L^\infty_{w^*} (0,T;\mathcal{M}( {\ov\Omega} ; \R^{d\times d}_{\sym,+}))$ and a pressure part
$\frakR_2\in L^\infty_{w^*} (0,T;\mathcal{M}( {\ov\Omega} ; \R_+)) $ such that 
\begin{subequations}
    \label{eq:dissweakEul}
\begin{align}
   - \int_\Omega \rho  \psi \,\dd x\bigg|_s^t +\int_s^t \int_\Omega \rho \, \partial_t \psi+ m \cdot\nabla \psi\,\dd{x} \,\dd{\tau} &= 0,
   \\
   \begin{split}
 - \int_\Omega m \cdot \varphi\,\dd x \bigg|_s^t   + \int_s^t \int_\Omega m \cdot \,\partial_t \varphi + \left(\frac{m \otimes m}{\rho}+ \rho^\gamma {I}_d\right):\nabla \varphi \,\dd{x}\,\dd \tau
 \\ +\int_0^T \Big(\int_{{{\ov\Omega}}} (\nabla \varphi)_{\sym} :\dd \frakR_1 + \int_{{{\ov\Omega}}} (\di \varphi) \,\dd \frakR_2\Big)\dd{\tau}   &= 0 ,
 \end{split}
\end{align}
holds for all test functions $(\psi, \varphi) \in C^1([0,T];C^1({\ov\Omega}) \times \mathcal{V}) $  and for a.a.~$s,t\in (0,T)$, $s<t$, including $s=0$ with $(\rho(0), m(0))= (\rho_0,m_0)$,
and if
\begin{equation}
    \frac{1}{2} \int_{{{\ov\Omega}}} \,\dd \tr ( \frakR_1)+ \frac{1}{\gamma-1} 
    \int_{{\ov\Omega}}\de \frakR_2 \leq E -  \mathcal{E}(\rho,m)
\end{equation}
\end{subequations}
holds a.e.~in $(0,T)$.
\end{dfn}

\begin{remark}[Existence of generalized solutions]
    In the case of periodic boundary conditions,
    the existence of energy-variational solutions was shown as an example of a general existence result~\cite{eiterlasarzik2024envar}.
    It was further generalized in~\cite{Marcel},
    which allows to also include the present case of a bounded domain.
    The existence of dissipative weak solutions was shown in~\cite[Prop.~1]{Westdickenberg23}. These works also investigate different selection criteria. 
\end{remark}

We show that both solution concepts are equivalent.

\begin{thm}\label{thm:EulerComp}
    Let $\Omega\subset\R^d$ be a bounded Lipschitz domain and $\gamma>1$. Then 
         $(\rho, \f m ,E)$ is an energy-variational solution in the sense of Definition~\ref{def:envarEuler}
         if and only if
        $(\rho, \f m , E)$ is  a dissipative weak solution in the sense of Definition~\ref{def:dissweakEuler}.
\end{thm}

\begin{proof}
    In order to prove the equivalence theorem, we want to Lemma~\ref{lem:equi}. 
    To this end, we consider the space $\mathbb Y=C^1([0,T];C^1({\ov\Omega})\times \mathcal{V})$  and  the functional $\ell : \mathbb Y \to \R$ given by 
    \begin{align*}
        \langle\ell, (\psi,\varphi)\rangle := {}& \left[ \int_\Omega \rho  \psi + m \cdot\,\varphi \,\dd{x} \right]\bigg|_0^T -\int_0^T \int_\Omega \rho \, \partial_t \psi + m \cdot\nabla \psi \, \dd x \dd t \\&-\int_0^T \int_{\Omega} m \cdot \,\partial_t \varphi + \left(\frac{m \otimes m}{\rho}+ \rho^\gamma {I}_d\right):\nabla \varphi \,\dd{x} \dd t %
 \,.
    \end{align*}
We choose the cones $\calC_1:=\{ z = (0, A)^T \in \R^{(1+d)\times d}\mid A \in \R^{d \times d}_{\sym,+}\}  $
and $\calC_2:= \{ z = (0, \alpha I_d)^T\in \R^{(1+d)\times d}\mid \alpha \in \R_+ \}$. These give rise to the linear spaces $\calL_1=\spa (\calC_1)$ and $\calL_2=\spa(\calC_2)$.
We follow the natural choices and equip 
the space $\calL_1$ with the spectral norm, \textit{i.e.,} 
        $ \|z\|_{\calL_1} := | A |_{2}$,
         and $\calL_2$ with the norm $ \| z \|_{\calL_2} := |\alpha|$.

Via a classical variational argument,
we observe that~\eqref{eq:envarEul} holds if and only if  $E\big|_s^t\leq 0$ for a.e.~$0\leq s<t\leq T$ 
and
$$
\begin{aligned}
    \langle \ell , (\psi,\varphi) \rangle &\leq \int_0^T \tilde{\mathcal{K} }(-\varphi) [E- \mathcal{E}(\rho , m)]\,\de t
\\
& = \int_0^T \max\{ 2\|(\nabla \varphi)_{\mathrm{sym},+}\|_{L^\infty(\Omega;\R^{d\times d})},(\gamma{-}1)\|(\di\varphi)_{+}\|_{L^\infty(\Omega)}\}[E- \mathcal{E}(\rho , m)] \,\dd t \,
\end{aligned}
$$
for all $\varphi \in \mathbb Y$. 
According to Lemma~\ref{lem:equi}, this inequality is equivalent to 
 the existence of $\tilde\frakR_1\in L^\infty_{w^*}(0,T;\mathcal{M}({\ov\Omega};\calC_1))$ and $\tilde\frakR_2\in L^\infty_{w^*}(0,T;\mathcal{M}({\ov\Omega};\calC_2))$ such that
        \begin{align}\label{eq:normfrakRK}
        \begin{split}
            \langle\ell , (\psi,\varphi)\rangle =\int_0^T \Big(\int_{{\ov\Omega}}
            (\nabla \psi, \nabla \varphi)^T : \dd \tilde\frakR_1 + \int_{{\ov\Omega}}
            (\nabla \psi, \nabla \varphi )^T: \dd\tilde\frakR_2\Big)\, \dd t
            &\quad\text{for all }\varphi\in\mathbb Y,
            \\
            \frac{1}{2}\| \tilde\frakR_1\|_{\mathcal{M}({\ov\Omega};\calL_1^*)} + \frac{1}{\gamma-1}\| \tilde\frakR_2\|_{\calM({\ov\Omega};\calL_2^*)} \leq \zeta
            &\quad\text{a.e.~in $(0,T)$} \,,
            \end{split}
        \end{align}
        where 
        $ \zeta= E-\E(\rho,m)$.  
        From the definition of the cones~$\calC_1$ and $\calC_2$, we infer that there exist two measures $ \frakR_1 \in L^\infty_{w^*}(0,T;\calM({\ov\Omega};\R^{d\times d}_{\sym,+}))$ and $\frakR_2 \in L^\infty_{w^*}(0,T;\calM({\ov\Omega}; \R_+))$ such that $\tilde\frakR_1$ and $\tilde\frakR_2$ can be represented as 
$$
\tilde\frakR_1:= 
   ( \f 0 ,
    \frakR_1 )^T
\, , \qquad \tilde\frakR_2:= 
  (  \f 0,
    \frakR_2 I )^T,
$$
that is, 
$$
\int_{{\ov\Omega}}  (\nabla \psi, \nabla \varphi)^T : \dd \tilde\frakR_1 = \int_{\ov\Omega} (\nabla \varphi)_{\sym} : \de \frakR_1 \,, \quad \int_{{\ov\Omega}}  (\nabla \psi, \nabla \varphi)^T : \dd \tilde\frakR_2 = \int_{\ov\Omega} \dv \varphi  \,\de \frakR_2 \,,
$$
        which allows to rewrite the first line in~\eqref{eq:normfrakRK}.
        To represent the second line in terms of $\frakR_1$ and $\frakR_2$, we have to take the total variation measure with respect to the correct dual norms.
        By Lemma~\ref{lem:norm}, 
        the norm on $\calL_1^*$ originates from the trace norm, which yields
        (as $\frakR$ is positive semi-definite)
        that 
        $ \| \tilde\frakR_1\|_{\mathcal{M}({\ov\Omega};\calL_1^*)} = \int_{{\ov\Omega}} \de |\frakR_1|_{\mathrm{tr}} = \int_{{\ov\Omega}} \de \tr (\frakR_1)$. Furthermore, the dual norm relevant for the total variation measure of $\frakR_2$ is the standard one for real-valued measures, based on the Euclidean norm.
        As $\frakR_2$ is a nonnegative measure,
        this yields the identity $\| \tilde\frakR_2\|_{\calM({\ov\Omega};\calL_2^*)}  = \int_{{\ov\Omega}} \de \frakR_2$. 
        Rewriting~\eqref{eq:normfrakRK} in terms of 
        $\frakR_1$ and $\frakR_2$,
        we thus obtain the equivalence of
        the formulations~\eqref{eq:envarEul} and~\eqref{eq:dissweakEul},
        which finishes the proof.        
\end{proof}

\section{The Euler--Korteweg system}
\label{sec:eulerkorteweg}

For $\Omega\subset\R^d$ a bounded Lipschitz domain 
and $T>0$, we consider the the Euler--Korteweg system 
given by
\begin{subequations}\label{au}
    \begin{align}
\partial_t \rho + \text{div} \; m &=0 \quad &&\text{in}\; \Omega \times (0,T) ,\\
   \partial_t m + \text{div}\left( \frac{m \otimes  m}{\rho}\right) + \nabla( \rho^\gamma)&=\rho \nabla\Delta \rho \quad &&\text{in} \;\Omega \times (0,T),  \\
   (\rho(0),m(0))&=(\rho_0,m_0) \quad &&\text{in} \;\Omega,
\end{align}
with adiabatic exponent $\gamma>1$, where for simplicity we have chosen the capillary coefficient to be 1.
 The system is equipped with impermeability boundary conditions for the momentum and homogeneous Neumann boundary conditions for the density,
\begin{align}
  m \cdot n&=0{} = \nabla\rho\cdot n \qquad \text{on } \partial \Omega\times(0,T).
\end{align}\end{subequations}

The total energy $\E:(H^1(\Omega) \cap L^{\gamma}(\Omega)) \times L^{\frac{2\gamma}{\gamma+1}}(\Omega;\R^d) \rightarrow [0,\infty]$ is defined as
$$\mathcal{E}(\rho,m)=\int_\Omega \eta(\rho(x),m(x))+\frac{|\nabla \rho(x)|^2}{2}\,\dd{x},$$ 
where the potential $\eta$ 
is given by~\eqref{eq:defeta}.
We set $\mathcal{V}=\{\varphi \in C^2(\overline{\Omega};\mathbb{R}^d)\mid n\,\cdot \varphi=0\text{ on } \partial\Omega\}$
and consider the regularity weight $\mathcal{K}:C^1({\ov\Omega} ;\R^d)\times \mathcal{V} \to [0,\infty)$ given by $\mathcal{K}(\psi,\varphi):= \tilde{\mathcal{K} }(\varphi)$
for the effective regularity weight $\tilde{\mathcal{K} }: \mathcal{V} \to [0,\infty)$
with
\begin{align}
      \tilde{\mathcal{K} }(\varphi):=\max\{ 2\|(\nabla \varphi)_{\mathrm{sym},-}\|_{L^\infty(\Omega;\R^{d\times d})} +\|(\di\varphi)_{-}\|_{L^\infty(\Omega)};  (\gamma-1)\|(\di\varphi)_{-}\|_{L^\infty(\Omega)}\}\,.
\end{align}
Energy-variational and dissipative weak solutions for the system~\eqref{au} are given in the following definitions.

\begin{dfn}\label{def:envarEulerKorte}
A triple $(\rho,m,E)\in L^\infty(0,T;H^1(\Omega) \cap L^\gamma(\Omega)) \times L^\infty(0,T;L^\frac{2\gamma}{\gamma +1}(\Omega;\R^d)) \times \mathrm{BV}([0,T])$ is called an \textit{energy-variational solution} to the Euler--Korteweg equations~\eqref{au} if $\mathcal{E}(\rho(t),m(t)) \leq E(t) $ for a.e.~$t \in (0,T)$ and  if the energy-variational inequality
\begin{equation}\label{eq:envarEulKor}
\begin{aligned}
    &\left[E- \int_\Omega \rho  \psi + m \cdot\,\varphi \,\dd{x} \right]\bigg|_s^t 
    \\
    &\quad+\int_s^t \int_\Omega \rho \, \partial_t \psi+ m \cdot\nabla \psi + m \cdot \,\partial_t \varphi + \left(\frac{m \otimes m}{\rho}+ \rho^\gamma {I}_d\right):\nabla \varphi \,\dd{x} \dd{\tau} \\
      &\quad+\int_s^t\int_\Omega  \rho\nabla \rho \cdot \nabla (\di
       \varphi) + 
 \frac{1}{2} |\nabla \rho|^2\di \varphi  + \nabla \rho \otimes \nabla \rho:\nabla \varphi \,\dd{x}\dd{\tau}\\
     &\quad+ \int_s^t \tilde{\mathcal{K} }(\varphi)(\mathcal{E}(\rho,m)-E) \,\dd{\tau}\leq 0
\end{aligned}
\end{equation}
holds for all test functions $(\psi, \varphi) \in C^1([0,T];C^1({\ov\Omega}) \times \mathcal{V}) $  and for a.a.~$s,t\in (0,T)$, $s<t$, including $s=0$ with $(\rho(0), m(0)) = (\rho_0,m_0)$.

\end{dfn}

\begin{dfn}\label{def:dissweakEulerKorte}
A triple $(\rho,m,E)\in L^\infty(0,T;H^1(\Omega) \cap L^\gamma(\Omega)) \times L^\infty(0,T;L^\frac{2\gamma}{\gamma +1}(\Omega;\R^d)) \times \mathrm{BV}([0,T])$ is called a \textit{dissipative weak solution} to the Euler--Korteweg equations~\eqref{au} if 
 $E$ is nonincreasing and if
there exists a Reynolds stress $\frakR_1 +\frakR_2I_d$  
consisting of 
$\frakR_1 \in L^\infty_{w^*} (0,T;\mathcal{M}( {\ov\Omega} ; \R^{d\times d}_{\sym,+}))$ 
and 
$\frakR_2\in L^\infty_{w^*} (0,T;\mathcal{M}( {\ov\Omega} ; \R_+)) $ such that \begin{subequations}
\begin{align}
   - \int_\Omega \rho  \psi \,\dd x\bigg|_s^t +\int_s^t \int_\Omega \rho \, \partial_t \psi+ m \cdot\nabla \psi\,\dd{x} \dd{\tau} &= 0,
   \\
   \begin{split}
 - \int_\Omega m \cdot \varphi \,\dd x\bigg|_s^t   + \int_s^t \int_\Omega m \cdot \,\partial_t \varphi + \left(\frac{m \otimes m}{\rho}+ \rho^\gamma {I}_d\right):\nabla \varphi \,\dd{x}\dd \tau \quad &
 \\
+\int_s^t\int_\Omega  \rho\nabla \rho \cdot \nabla (\di\varphi) + 
 \frac{1}{2} |\nabla \rho|^2\di \varphi  + \nabla \rho \otimes \nabla \rho:\nabla \varphi \,\dd{x}\dd{\tau} \ \ &\\
 +\int_s^t \Big(\int_{{{\ov\Omega}}} (\nabla \varphi)_{\sym} :\dd \frakR_1 + \int_{{{\ov\Omega}}} (\di \varphi) \,\dd \frakR_2\Big)\dd{\tau}&= 0 ,
 \end{split}\label{eq:EulKormom}
 \end{align}
 holds for all test functions $(\psi, \varphi) \in C^1([0,T];C^1({\ov\Omega}) \times \mathcal{V}) $  and for a.a.~$s,t\in (0,T)$, $s<t$, including $s=0$ with $(\rho(0), m(0))= (\rho_0,m_0)$,
 and if
 \begin{equation}
     \int_{{{\ov\Omega}}} \dd \max\left\{\frac{1}{2}\tr ( \frakR_1); | \frakR_1|_2\right\}+ \frac{1}{\gamma-1} 
    \int_{{\ov\Omega}}\de \frakR_2 \leq E -  \mathcal{E}(\rho,m)\label{eq:EulKorEn}\,
\end{equation}
\end{subequations}
holds a.e.~in $(0,T)$.
\end{dfn}

\begin{remark}
\label{rem:measure.norm}
Since $\frakR_1(t)$ is an $\R^{d\times d}_{\sym,+}$-valued measure,
the first term on the right-hand side of~\eqref{eq:EulKorEn} 
can be written as
\[
\int_{{{\ov\Omega}}} \dd \max\left\{\frac{1}{2}\tr ( \frakR_1); | \frakR_1|_2\right\}
=\int_{{{\ov\Omega}}} \dd |\frakR_1|_m 
= |\frakR_1|_m (\ov\Omega),
\]
where $|\frakR_1|_m$ denotes the variation of the measure $\frakR_1$ with respect to the norm
\[
|A|_m=\max\left\{\frac{1}{2}| A|_{\tr}; | A|_2\right\}.
\]
By the Riesz representation theorem, we have
\[
 |\frakR_1|_m (\ov\Omega)=\sup\setcL{\int_{\ov\Omega}\Phi:\dd\frakR_1}{\Phi\in C(\ov\Omega;\R^d), \,|\Phi|_m^*\leq 1\text{ in }\ov\Omega}.
\]
with the dual norm
$|A|_m^*=2|A|_{2}+|A|_{\tr}$, compare Lemma~\ref{lem:norm}.
\end{remark}

\begin{remark}[Existence of generalized solutions]
    The existence of energy-variational solutions to~\eqref{au} was recently shown as an application of a general existence result~\cite{Marcel}. 
    While there seems to be no result for the existence of dissipative weak solutions, 
    the existence of dissipative measure-valued solutions was shown in~\cite{GGSW}. 
    For additional literature on the existence theory for the Euler--Korteweg system, see~\cite{audiardhaspot2017eulerkorteweg,giesselmannlattanziotzavaras2017eulerkorteweg,benzonigavagedanchindescombes2007eulerkorteweg}.
\end{remark}

As before, we show that the two solution concepts for the Euler--Korteweg system are equivalent.

\begin{thm}
    Let $\Omega\subset\R^d$ be a bounded Lipschitz domain and $\gamma>1$. Then 
       $(\rho, \f m ,E)$ is   an energy-variational solution in the sense of Definition~\ref{def:envarEulerKorte} 
        if and only if 
         $(\rho, \f m , E)$ is a dissipative weak solution in the sense of Definition~\ref{def:dissweakEulerKorte}.
\end{thm}

\begin{proof}
    We follow the proof of Theorem~\ref{thm:EulerComp} and apply Lemma~\ref{lem:equi}. 
    To this end, we define the space $\mathbb Y=C^1([0,T];C^1({\ov\Omega})\times \mathcal{V})$  and the functional $\ell : \mathbb Y \to \R$ with
    \begin{align*}
        \langle\ell, (\psi,\varphi)\rangle := {}& \left[ \int_\Omega \rho  \psi + m \cdot\,\varphi \,\dd{x} \right]\bigg|_0^T -\int_0^T \int_\Omega \rho \, \partial_t \psi + m \cdot\nabla \psi  \,\dd x \dd t \\&-\int_0^T \int_{\Omega} m \cdot \,\partial_t \varphi + \left(\frac{m \otimes m}{\rho}+ \rho^\gamma {I}_d\right):\nabla \varphi \,\dd{x} \dd t \\
      &-\int_0^T\int_\Omega  \rho\nabla \rho \cdot \nabla (\di
       \varphi) + 
 \frac{1}{2} |\nabla \rho|^2\di \varphi  + \nabla \rho \otimes \nabla \rho:\nabla \varphi \,\dd{x}\dd t  \,.
    \end{align*}
We define the cones $\calC_1:=\{ z = (0, A)^T\in \R^{(1+d)\times d}\mid A \in \R^{d \times d}_{\sym,+}  \}  $
and $\calC_2:= \{ z = (0, \alpha I_d)^T\in \R^{(1+d)\times d}\mid \alpha \in \R_+ \}$
and the associated linear spaces $\calL_1:=\spa(\calC_1)$ and $\calL_2:=\spa(\calC_2)$
as before. 
While we again equip $\calL_2$ with the norm $ \| z \|_{\calL_2} := |\alpha|$,
we use a different norm on the first space~$\calL_1$,
where we consider a weighted sum of the spectral norm and the trace norm, \textit{i.e.,} 
        $ \|z\|_{\calL_1} :=2 | A |_{2} +| A |_{\tr}$.
         Due to Lemma~\ref{lem:norm}, the dual norm of $\|\cdot\|_{\calL_1}$ with respect to the Frobenius product is given by $\|z\|_{\calL_1^*} =  \max\{1/2| A |_{\tr}, |A|_{2}\}$
         for $z=(0,A)^T$.
         
By the same scaling argument as before, it can be observed that~\eqref{eq:envarEulKor} holds if and only if $E\big|_s^t\leq 0$ for a.e.~$0\leq s<t\leq T$ as well as 
$$
\begin{aligned}
 \langle &\ell , (\psi,\varphi) \rangle \leq \int_0^T \tilde{\mathcal{K} }(-\varphi) [E- \mathcal{E}(\rho , m)]\de t
 \,
\end{aligned}
$$
for all $(\psi,\varphi)\in \bbY$. 
According to Lemma~\ref{lem:equi}, the above inequality is equivalent to 
the existence of $\tilde\frakR_1\in L^\infty_{w^*}(0,T);\mathcal{M}({\ov\Omega};\calC_1))$ and $\tilde\frakR_2\in L^\infty_{w^*}(0,T);\mathcal{M}({\ov\Omega};\calC_2))$ such that
        \begin{align}\label{eq:normfrakRKort}
        \begin{split}
            \langle\ell , (\psi,\varphi)\rangle =\int_0^T\Big( \int_{{\ov\Omega}}
            (\nabla \psi, \nabla \varphi)^T : \dd \tilde\frakR_1 + \int_{{\ov\Omega}}
            (\nabla \psi, \nabla \varphi )^T: \dd\tilde\frakR_2 \Big)\dd t \quad &\text{for all }\varphi\in\mathbb Y,
            \\  
            \| \tilde\frakR_1\|_{\mathcal{M}({\ov\Omega};\calL_1^*)} + \frac{1}{\gamma-1}\| \tilde\frakR_2\|_{\calM({\ov\Omega};\calL_2^*)} \leq \zeta \quad &\text{a.e.~in $(0,T)$} \,,
            \end{split}
        \end{align}
        where $ \zeta= E-\E(\rho,m)$.
        Arguing as before, we obtain 
        $
\tilde\frakR_1:= 
   ( \f 0 ,
    \frakR_1 )^T$
and 
$\tilde\frakR_2:= 
  (  \f 0,
    \frakR_2 I )^T
$ for some elements
$\frakR_1 \in L^\infty_{w^*}(0,T;\calM({\ov\Omega};\R^{d\times d}_{\sym,+}))$ and $\frakR_2 \in L^\infty_{w^*}(0,T;\calM({\ov\Omega}; \R_+))$,
and it holds  
         $ \| \frakR_1\|_{\mathcal{M}({\ov\Omega};\calL_1^*)} = \int_{{{\ov\Omega}}} \dd \max\left\{\frac{1}{2}\tr ( \frakR_1); | \frakR_1|_2\right\}$ 
         and $\| \frakR_2\|_{\calM({\ov\Omega};\calL_2^*)}  = \int_{{\ov\Omega}} \de \frakR_2$.
         In virtue of these identities,
         we obtain the equivalnce of~\eqref{eq:normfrakRKort} and~\eqref{eq:EulKormom}.
         In total, this yields the asserted equivalence of energy-variational and dissipative weak solutions to~\eqref{au}.   
\end{proof}

\section{The Euler--Poisson equations}
\label{sec:eulerpoisson}

For $\Omega\subset\R^d$ a bounded Lipschitz domain 
and $T>0$, we consider the the Euler--Poisson system 
given by
\begin{subequations}\label{eq:EulerPoisson}
\begin{align}
     \partial_t \rho + \di m &= 0 \quad &&\text{in}\; \Omega \times (0,T) ,\label{eq:EulPoi1} \\
     \partial_t m + \di \Big(\frac{m \otimes m}{\rho}\Big)+ \nabla \rho^\gamma&= -\alpha \, m 
     -\rho \nabla V \quad &&\text{in}\; \Omega \times (0,T) , \label{eq:EulPoi2}\\
    - \Delta V &= \rho - \overline{\rho} \quad &&\text{in}\; \Omega \times (0,T)\label{eq:EulPoi3} , \\
       (\rho(0),m(0))&=(\rho_0,m_0) \quad &&\text{in} \;\Omega \, ,
       \label{eq:EulPoi4}
\end{align}
where $ \bar\rho:= \frac{1}{|\Omega|}\int_{\Omega} \rho \,\dd x $ denotes the mean of $\rho$.
Here we follow the sign convention from~\cite{dissweakEulPoisson},
and $\alpha\geq 0$ is a friction coefficient. 
The system is equipped with boundary conditions 
\begin{align}
  m \cdot n&=0 \quad \text{on} \quad \partial \Omega \times (0,T),  \\
 \nabla V\cdot n&=0 \quad \text{on} \quad \partial \Omega \times (0,T).
\end{align}
\end{subequations}

The term $V$ in~\eqref{eq:EulPoi2} can be interpreted as the occurrence of a nonlocal operator of $\rho$ via $ V = (-\Delta)^{-1} (\rho-\bar\rho) $.
Here, $(-\Delta)^{-1} : H^{-1}_{(0)}(\Omega) \to H^1_{(0)}(\Omega)$ denotes the inverse of the Neumann--Laplacian, where 
$ H^1_{(0)}(\Omega)$ is the space of functions $u\in H^1(\Omega)$ with vanishing mean, \textit{i.e.,} such that $ \int_{\Omega}u \,\de x = 0$, and $ H^{-1}_{(0)}(\Omega)$ is its dual space.
More precisely, for $h\in H^{-1}_{(0)}(\Omega)$, 
we have $u=(-\Delta)^{-1}h$ if and only if $u$ is the unique solution to 
$$
\int_{\Omega} \nabla u \cdot \nabla \varphi \,\de x = \langle h , \varphi \rangle \quad \text{ for all } \varphi \in H^1_{(0)} (\Omega)\,.
$$

We define the total energy $\E:\big(L^{\gamma}(\Omega)\cap (H^1(\Omega))^*\big) \times L^{\frac{2\gamma}{\gamma+1}}(\Omega;\R^d) \rightarrow [0,\infty]$ as 
$$\mathcal{E}(\rho,m)=\int_\Omega \eta(\rho(x),m(x))+\frac{1}{2}|\nabla V|^2\dd{x},
\quad \text{where }V = (-\Delta)^{-1} (\rho-\bar\rho),$$ 
for $\eta$ given by~\eqref{eq:defeta}. 
We set $\mathcal{V}=\{\varphi \in C^1(\overline{\Omega};\mathbb{R}^d): n\,\cdot \varphi=0\text{ on }\partial\Omega\}$
and consider the regularity weight $\mathcal{K}:C^1({\ov\Omega} ;\R^d)\times \mathcal{V} \to [0,\infty)$ given by $\mathcal{K}(\psi,\varphi):= \tilde{\mathcal{K} }(\varphi)$
for the effective regularity weight $\tilde{\mathcal{K} }: \mathcal{V} \to [0,\infty)$
with
\begin{align}\label{eq:regweightEulPoi}
      \mathcal{K}(\psi,\varphi):=\tilde{\mathcal{K} }(\varphi):=\max\{
      (2+d)\|(\nabla \varphi)_{\mathrm{sym}}\|_{L^\infty(\Omega;\R^{d\times d})}, (\gamma-1) \|(\di\varphi)_{-}\|_{L^\infty(\Omega)}\}
        \,.
\end{align}
Observe that~\eqref{eq:regweightEulPoi} is similar to the regularity weight~\eqref{eq:regweight.Euler.compr} for the isentropic Euler equations,
where only the negative semidefinite part of $(\nabla\varphi)_{\sym}$ contributes to $\mathcal K$.
In the present situation, one also has to take into account the positive semidefinite part
in order to ensure the existence of energy-variational solutions, 
compare~\cite{Marcel} and Remark~\ref{rem:regweight.ep.finer} below.

In order to derive a reasonable weak solution concept, we reformulate
the Poisson term in~\eqref{eq:EulPoi2} as
$$
-\rho \nabla V = - (\rho - \bar{\rho}) \nabla V  - \bar\rho\nabla  V = \Delta V \nabla V  - \bar\rho \nabla V = \di ( \nabla V \otimes \nabla V ) - \frac{1}{2}\nabla | \nabla V |^2 - \bar\rho \nabla V \,,
$$
where we made use of~\eqref{eq:EulPoi3} and the chain rule. 
We define energy-variational solutions to~\eqref{eq:EulerPoisson} as follows.
\begin{dfn}\label{def:envarEulerPoisson}
A triple 
\[
(\rho,m,E)\in L^\infty(0,T;(H^1(\Omega))^* \cap L^\gamma(\Omega)) \times L^\infty(0,T;L^\frac{2\gamma}{\gamma +1}(\Omega;\R^d)) \times \mathrm{BV}([0,T])
\]
is called an \textit{energy-variational solution} to the Euler--Poisson system~\eqref{eq:EulerPoisson} if it holds $\mathcal{E}(\rho(t),m(t)) \leq E(t) $ for a.e.~$t \in (0,T)$ and  if the energy-variational inequality
\begin{equation}
    \begin{aligned}
    &\left[E- \int_\Omega \rho  \psi + m \cdot\,\varphi \,\dd{x} \right]\bigg|_s^t 
    \\
    &\quad+\int_s^t \int_\Omega \rho \, \partial_t \psi+ m \cdot\nabla \psi + m \cdot \,\partial_t \varphi + \left(\frac{m \otimes m}{\rho}+ \rho^\gamma {I}_d\right):\nabla \varphi 
    - \alpha m\cdot\varphi\,\dd{x} \dd{\tau} \\
      &\quad+\int_s^t\int_\Omega \alpha \frac{|m|^2}{\rho} +\left(\frac{1}{2} |\nabla V|^2  +\bar{\rho} V \right) \di \varphi -\nabla V \otimes \nabla V : \nabla \varphi \,\dd{x}\dd{\tau} \\
     &\quad+ \int_s^t  \tilde{\mathcal{K} }(\varphi) (\mathcal{E}(\rho,m)-E) \,\dd{x}\dd{\tau}\leq 0
\end{aligned}
\end{equation}
holds for all test functions $(\psi, \varphi) \in C^1([0,T];C^1({\ov\Omega}) \times \mathcal{V}) $
and for a.a.~$s,t\in (0,T)$, $s<t$, including $s=0$ with  $(\rho(0), m(0)) = (\rho_0,m_0)$.
\end{dfn}

\begin{remark}[Choice of the regularity weight]
\label{rem:regweight.ep.finer}
    Another possible and finer choice for the regularity weight would be given by
    $$\begin{aligned}
        \mathcal{K}_2(\varphi) := \max\{&2 \| (\nabla \varphi)_{\sym,-}\|_{L^\infty(\Omega;\R^{d\times d})}; 2 \| (\nabla \varphi)_{\sym,+}\|_{L^\infty(\Omega;\R^{d\times d})}+ \| (\di \varphi)_-\|_{L^\infty(\Omega)}\\&;(\gamma-1)  \| (\di \varphi)_-\|_{L^\infty(\Omega)} \}\,.
    \end{aligned}
    $$
Indeed, via elementary estimates, we find 
\begin{align*}
   \int_\Omega& \left(\frac{m \otimes m}{\rho}+ \rho^\gamma {I}_d\right):\nabla \varphi \dd{x} \dd{\tau}+ \left(\frac{1}{2} |\nabla V|^2  \right) \di \varphi -\nabla V \otimes \nabla V : \nabla \varphi  \dd{x}
   \\
   \geq{}& - 2 \| (\nabla \varphi)_{\sym,-}\|_{L^\infty(\Omega;\R^{d\times d})} \int_{\Omega} \frac{|m|^2}{2\rho}\de x - (\gamma-1) \|(\di \varphi)_-\|_{L^\infty(\Omega)} \int_{\Omega} \frac{\rho^\gamma}{\gamma-1} \de x\\&- \left( \|(\di \varphi)_-\|_{L^\infty(\Omega)} + 2 \| (\nabla \varphi)_{\sym,+}\|_{L^\infty(\Omega;\R^{d\times d})} \right) \frac{1}{2}\| \nabla V\|_{L^2(\Omega)}^2
  \\ \geq{}& -\mathcal{K}_2(\varphi) \mathcal{E}(\rho, m)
   \,.
\end{align*}
    Therefore, the nonlinearities are dominated and convexified by $ \mathcal{K}_2(\varphi)\mathcal{E}(\rho,m)$ in the usual way. 
    This is a finer choice than $\tilde{\mathcal{K}}$ above, and it fits into the general assumptions studied in~\cite{Marcel} or~\cite{abramo}. For the equivalence result presented here, we choose the regularity weight $\tilde{\mathcal{K}}$ as it fits better to the previous framework based on Lemma~\ref{lem:equi}. From the norm estimate $ |\cdot|_{\tr} \leq d |\cdot|_2$,
    we find $\mathcal{K}_2(\varphi)\leq\tilde{\mathcal{K}}(\varphi) $
    for all $\varphi\in\calV$. Therefore,  every energy-variational solution for the regularity weight~$\mathcal{K}_2$ is also an energy-variational solution for the regularity weight~$\tilde{\mathcal{K}}$. 
\end{remark}
\begin{dfn}\label{def:dissweakEulerPoisson}
A triple 
\[
(\rho,m,E)\in L^\infty(0,T;H^{-1}(\Omega) \cap L^\gamma(\Omega)) \times L^\infty(0,T;L^\frac{2\gamma}{\gamma +1}(\Omega;\R^d)) \times \mathrm{BV}([0,T])
\]
is called a \textit{dissipative weak solution} to the Euler--Poisson system~\eqref{eq:EulerPoisson} if 
there exists a Reynolds stress $\frakR_1 +\frakR_2I_d$ consisting of
$\frakR _1 \in L^\infty_{w^*} (0,T;\mathcal{M}(\ov \Omega ; \R^{d\times d}_{\sym}))$ and 
$ \frakR_2 \in L^\infty_{w^*} (0,T;\mathcal{M}(\ov \Omega ; \R^{+})) $ such that 
\begin{subequations}
\begin{align}
E \Big|_s^t + \int_s^t \int_{\Omega} \alpha \frac{|m|^2 }{\rho} \, \dd x \dd \tau & \leq 0,\\
   - \int_\Omega \rho  \psi \,\dd x\bigg|_s^t +\int_s^t \int_\Omega \rho \, \partial_t \psi+ m \cdot\nabla \psi\,\dd{x} \dd{\tau} &= 0,
   \\
   \begin{split}
 - \int_\Omega m \cdot \varphi \,\dd x\bigg|_s^t   + \int_s^t \int_\Omega m \cdot \,\partial_t \varphi + \left(\frac{m \otimes m}{\rho}+ \rho^\gamma {I}_d\right):\nabla \varphi -\alpha m\cdot\varphi\,\dd{x}\dd{\tau} 
 \\
      +\int_s^t\int_\Omega \left(\frac{1}{2} |\nabla V|^2  +\bar{\rho} V \right)  \di \varphi - \nabla V \otimes \nabla V : \nabla \varphi \,\dd{x} \dd{\tau}
      \\
      +\int_s^t\int_{\overline{\Omega}} (\nabla \varphi)_{\sym} :\dd \frakR _1+\int_{\overline{\Omega}}\di \varphi  \,\dd \frakR_2 \dd{\tau} &= 0 ,
     \end{split} 
\end{align}
holds for all test functions $(\psi, \varphi) \in C^1([0,T];C^1({\ov\Omega}) \times \mathcal{V}) $
and for a.a.~$s,t\in (0,T)$, $s<t$, including $s=0$ with $(\rho(0), m(0))= (\rho_0,m_0)$,
and if
\begin{equation}
    \frac{1}{2} \int_{\overline{\Omega}} \dd | \frakR_1|_{\tr} + \frac{1}{(\gamma-1) } \int_{{\ov\Omega}} \de \frakR_2 \leq E-  \mathcal{E}(\rho,m)
\end{equation}
\end{subequations}
holds a.e.~in $(0,T)$.
\end{dfn}

\begin{thm}
    Let $\Omega$ be a bounded Lipschitz domain and $\gamma>1$. Then   $(\rho, \f m ,E)$ is an  energy-variational solution in the sense of Definition~\ref{def:envarEulerPoisson} if and only if    $(\rho, \f m , E)$ is a dissipative weak solution in the sense of Definition~\ref{def:dissweakEulerPoisson}.\label{pr:dissEulerPoisson}
\end{thm}
\begin{proof}
We can follow the lines of the proofs in the previous two sections, which are very similar. 
This is why we only focus on the differences here. 

We now choose the cones $\calC_1:=\{ z = (0, A)^T\in \R^{(1+d)\times d}\mid A \in \R^{d \times d}_{\sym} \}  $
and $\calC_2:= \{ z = (0, \alpha I)^T\in \R^{(1+d)\times d}\mid \alpha \in \R_+ \}$
with associated linear subspaces $\calL_1:=\spa(\calC_1) = \calC_1$ and $\calL_2:=\spa(\calC_2)$.
Note that the first cone is a full linear subspace, in contrast to the previous examples.  
We equip $\calL_1$ with the same norm as in the compressible Euler case in Section~\ref{sec:compEul}, namely $\|z\|_{\calL_1} := \| A \|_{2} $. On $\calL_2$, we take the usual norm and set $ \|z\|_{\calL_2} := | z | $. Defining 
$$
\begin{aligned}
     \langle \ell, (\psi,\varphi)\rangle ={}&- \int_\Omega \rho  \psi + m \cdot\,\varphi \,\dd{x} \bigg|_0^T \\&+\int_0^T \int_\Omega \rho \, \partial_t \psi+ m \cdot\nabla \psi + m \cdot \,\partial_t \varphi + \left(\frac{m \otimes m}{\rho}{+} \rho^\gamma {I}_d\right):\nabla \varphi 
     -\alpha m\cdot\varphi\,\dd{x} \dd{t} \\
      &+\int_0^T\int_\Omega\left(\frac{1}{2} |\nabla V|^2 +\bar{\rho} V \right) \di \varphi -\nabla V \otimes \nabla V : \nabla \varphi     \,\dd{x}\dd{t} ,
\end{aligned}
$$
we may apply Lemma~\ref{lem:equi} in order to observe the equivalence of 
Definition~\ref{def:envarEulerPoisson} and Definition~\ref{def:dissweakEulerPoisson}. 
In contrast to before, the cone $\calC_1$ does not only consist of elements originating from positive definite matrices,
so that the trace norm $|\cdot|_{\mathrm{tr}}$ does not simplify to the trace of the matrix.  
\end{proof}

\section*{Acknowledgement}
The authors acknowledge financial support by the German Research Foundation (DFG), through grant SPP 2410 ``Hyperbolic Balance Laws in Fluid Mechanics:~Complexity, Scales, Randomness (CoScaRa)'',
Projects No.~526018747 and No.~525716336.
T.~Eiter's research has further been funded by DFG through grant CRC 1114 ``Scaling Cascades in Complex Systems'', Project No.~235221301, Project YIP.

\end{document}